\documentclass[10pt,a4paper]{amsart}
\usepackage[latin1]{inputenc}
\usepackage{amsmath}
\usepackage{amsfonts}
\usepackage{amssymb}
\usepackage{bm}
\usepackage{url} 
\usepackage{proof}
\usepackage{graphicx}
\usepackage{subfig}
\author{Oliver J. D. Barrowclough}
\author{Tor Dokken}
\email{oliver.barrowclough@sintef.no, tor.dokken@sintef.no}
\address{SINTEF ICT, Applied Mathematics\\P.O. Box 124, Blindern \\ 0314 Oslo, Norway}

\title{Approximate Implicitization Using Linear Algebra}

\newtheorem{thm}{Theorem}

\newtheorem{cor}[thm]{Corollary}
\newtheorem{prop}[thm]{Proposition}
\newtheorem{algorithm}{Algorithm}

\begin{document}

\begin{abstract}
In this paper we consider a family of algorithms for approximate implicitization of rational parametric curves and surfaces. The main approximation tool in all of the approaches is the singular value decomposition, and they are therefore well suited to floating point implementation in computer aided geometric design (CAGD) systems. We unify the approaches under the names of commonly known polynomial basis functions, and consider various theoretical and practical aspects of the algorithms. We offer new methods for a least squares approach to approximate implicitization using orthogonal polynomials, which tend to be faster and more numerically stable than some existing algorithms. We propose several simple propositions relating the properties of the polynomial bases to their implicit approximation properties.
\end{abstract}

\keywords{implicitization, approximation, change of representation, orthogonal polynomials}

\maketitle

\section{Introduction}

Implicitization algorithms have been studied in both the CAGD and algebraic geometry communities for many years. Traditional approaches to implicitization have focused on exact methods such as Gr\"obner bases, resultants and moving curves and surfaces, or syzygies \cite{sederberg_1995}. Approximate methods that are particularly well suited to floating point implementation have also emerged in the past 25 years \cite{bajaj_1993,chuang_1989,corless_2001,dokken_1997,pratt_1987}. These methods are closely related to the algorithms we present, however, those that fit most closely into the framework of this paper include \cite{corless_2001,dokken_2001,dokken_2006,wang_2004}.

Implicitization is the conversion of parametrically defined curves and surfaces into curves and surfaces defined by the zero set of a single polynomial. Exact implicit representations of rational parametric manifolds often have very high polynomial degrees, which can cause numerical instabilities and slow floating point calculations. In cases where the geometry of the manifold is not sufficiently complicated to justify this high degree, approximation is often desirable. Moreover, for CAGD systems based on floating point arithmetic, exact implicitization methods are often unfeasible due to performance issues. The methods we present, attempt to find `best fit' implicit curves or surfaces of a given degree $m$ (the definition of `best fit' varies with regard to the chosen method of approximation). One important property of all the algorithms is that they are guaranteed to give exact implicitizations for sufficiently high implicit degrees, up to numerical stability. In addition, some of the methods are also suitable for implementation in exact arithmetic, hence constituting alternative methods for exact implicitization. 

For simplicity of notation, we proceed for the majority of the paper to describe the implicitization of curves. In Sections \ref{sec:prelim}, \ref{sec:orig} and \ref{sec:weak} we introduce the notation and review existing methods. In Section \ref{sec:orthog} we present a new method for approximate implicitization using orthogonal polynomials and prove a theoretical relation to the previous methods. Implementations of the methods using different basis functions will be presented in Section \ref{sec:bases} and a qualitative comparison and discussion given in Section \ref{sec:comparison}. Finally, the generalization to both tensor-product and triangular surfaces will be covered in Section \ref{sec:surfaces}.

\section{Preliminaries}\label{sec:prelim}

A parametric curve in $\mathbb{R}^2$ is given by $\mathbf{p}(t) = (p_1(t),p_2(t))$ where $p_1$ and $p_2$ are functions in $t$ on some parameter domain $\Omega.$ Of particular importance both in CAGD and classical algebraic geometry are rational parametric curves (i.e., where $p_1$ and $p_2$ are rational functions). In the majority of this paper we will thus restrict our attention to planar rational curves, where the domain of interest is $\Omega=[0,1].$ In order to use polynomial bases in our construction, we can use the representation of the curves in the projective plane $\mathbb{P}^2.$ For a rational parametric curve $\mathbf{p}(t) = (g_1(t)/h(t),g_2(t)/h(t))$ in $\mathbb{R}^2,$ where $g_1,g_2$ and $h$ are polynomials, we thus use the homogeneous description
\[
\mathbf{p}(t) = (g_1(t),g_2(t),h(t)).
\]

All the methods to be described require a choice of degree $m$ and a choice of basis $(q_k(\mathbf{u}))_{k=1}^M,$ for the implicit polynomial. Here $M$ is defined as the number of basis functions in a polynomial of total degree $m.$ Thus, for a general bivariate polynomial, we have $M = \binom{m+2}{2}.$ Any polynomial $q$ can be written in terms of such a basis by choosing coefficients $\mathbf{b} = (b_k)_{k=1}^M:$
\begin{equation}\label{eq:q}
q(\mathbf{u}) = \sum_{k=1}^M b_k q_k(\mathbf{u}).
\end{equation}
The choice of implicit basis is an important factor which has implications for both the stability of the algorithms and the quality of the approximations. However, most of the work in this paper is independent of the choice of implicit basis. In $\mathbb{R}^2,$ a good choice is the Bernstein basis in a barycentric coordinate system defined over a triangle containing the parametric curve. For curves in $\mathbb{P}^2,$ we use the homogeneous Bernstein basis given by
\[
q_\mathbf{k}(u,v,w) = \binom{m}{k_1,k_2,k_3} u^{k_1} v^{k_2} w^{k_3}, \quad \text{ for } |\mathbf{k}|=k_1+k_2+k_3 = m,
\]
where $u,v$ and $w$ denote the homogeneous coordinates and $\mathbf{k}=(k_1,k_2,k_3)$ denotes a multi-index. We order the basis by letting $q_k$ correspond to $q_{\mathbf{k}}$ for $k=1,\ldots,M,$ where $k=k(\mathbf{k})$ denotes lexicographical order on the indices $k_1,k_2$ and $k_3.$ Unless otherwise stated, we will assume that the implicit basis $(q_k(\mathbf{u}))_{k=1}^M$ is the Bernstein basis. In particular, it forms a partition of unity
\[
\sum_{k=1}^M q_k(\mathbf{u}) \equiv 1.
\]
The choice of degree $m$ determines the number of degrees of freedom the implicit curve will have. If the chosen degree $m$ is sufficiently high, all the algorithms will produce exact results, up to numerical stability. If the degree is lower than necessary, approximations are produced.

Since we are searching for implicit representations and want to avoid the trivial solution $q\equiv 0,$ we add a normalization constraint to $q$ in the approximation. How best to make this choice has been discussed by several authors (see \cite{pratt_1987} for an overview). However, as we use the singular value decomposition (SVD) as the means of approximation, our results are given with the quadratic normalization $\Vert\mathbf{b}\Vert_2=1.$ 

The techniques in this paper focus on minimization of the objective function $q\circ\mathbf{p}$ over the space of polynomials $\{q:\Vert\mathbf{b}\Vert_2=1\},$ where $q$ is defined by (\ref{eq:q}) in a fixed implicit basis. Such a minimization, although not directly minimizing the Hausdorff distance between the implicit and parametric curves, is closely related to the geometric approximation problem. It has been shown that minimization of $q\circ\mathbf{p}$ gives excellent results in geometric space away from singularities \cite{dokken_1997}.

\section{Approximate implicitization - the original approach}\label{sec:orig}

In 1997 a class of techniques for approximate implicitization of rational parametric curves, surfaces and hypersurfaces was introduced in \cite{dokken_1997}. The approximation quality of the techniques was substantiated by a general proof that the methods exhibit very high convergence rates, as shown in Tables \ref{tab:conv_rates_curves} and \ref{tab:conv_rates_surfaces}. Extensions of this original approach, all of which inherit these high convergence rates, form the basis of this paper. 

\begin{table}
\begin{center}
\begin{tabular}{|r|c|c|c|c|c|c|c|c|}
  \hline
   Algebraic degree $m$  & 1 & 2 & 3 & 4  & 5  & 6  & 7  & 8 \\ \hline
   Convergence rate $k$ & 2 & 5 & 9 & 14 & 20 & 27 & 35 & 44 \\ 
  \hline
\end{tabular}
\caption{Convergence rates for approximate implicitization of sufficiently smooth parametric curves in $\mathbb{R}^2$ by algebraic curves of degree $m,$ given by $k=\frac{1}{2}(m+1)(m+2)-1$.}\label{tab:conv_rates_curves}
\end{center}
\end{table}

\begin{table}
\begin{center}
\begin{tabular}{|r|c|c|c|c|c|c|c|c|}
  \hline
   Algebraic degree $m$  & 1 & 2 & 3 & 4 & 5  & 6  & 7  & 8 \\ \hline
   Convergence rate $k$ & 2 & 3 & 5 & 7 & 10 & 12 & 14 & 17 \\ 
  \hline
\end{tabular}
\caption{Convergence rates for approximate implicitization of sufficiently smooth parametric surfaces in $\mathbb{R}^3$ by algebraic surfaces of degree $m,$ given by ${k=\lfloor\frac{1}{6}\sqrt{9+12m^3+72m^2+132m}-\frac{1}{2}\rfloor}.$}\label{tab:conv_rates_surfaces}
\end{center}
\end{table}

The guiding principle behind these methods is to find a polynomial $q$ which minimizes the maximal algebraic error in a given parameter domain $\Omega.$ That is, in the given implicit basis $(q_k)_{k=1}^M,$ to find the coefficients $\mathbf{b}=(b_k)_{k=1}^M$ which solve
\begin{equation}\label{eq:prob}
\min_{\Vert \mathbf{b}\Vert=1}\max_{t\in\Omega}|q(\mathbf{p}(t))|.
\end{equation}
The solution to this problem, which we call the minimax (or uniform) approximation, is not easy to find exactly. However, approximations to the minimax solution can be found directly, using linear algebra. 

We notice that the expression $q(\mathbf{p}(t))$ is a univariate polynomial of degree $mn$ in $t.$ We can thus approach the problem by first factorizing the error expression in a polynomial basis $\bm\alpha(t) = (\alpha_j(t))_{j=1}^{L},$ where $L=mn+1,$ as follows:
\begin{equation}\label{eq:fact}
q(\mathbf{p}(t)) = \bm\alpha(t)^T\mathbf{D}_\alpha\mathbf{b},
\end{equation}
where $\mathbf{D}_\alpha$ is a matrix whose columns are the coefficients of $q_k(\mathbf{p}(t))$ expressed in the $\bm\alpha$-basis and $\mathbf{b}$ is the unknown vector of implicit coefficients. Now, we have
\begin{align}\label{eq:ineq}
\min_{\Vert\mathbf{b}\Vert=1}\max_{t\in\Omega}|q(\mathbf{p}(t))| & = \min_{\Vert\mathbf{b}\Vert=1}\max_{t\in\Omega}|\bm\alpha(t)^T\mathbf{D}_\alpha\mathbf{b}|, \nonumber \\
& \leq \max_{t\in\Omega} \Vert \bm\alpha(t)\Vert_2\min_{\Vert\mathbf{b}\Vert=1}\Vert\mathbf{D}_\alpha\mathbf{b}\Vert_2, \nonumber \\
& = \max_{t\in\Omega} \Vert \bm\alpha(t)\Vert_2 \ \sigma_{\min},
\end{align}
giving an upper bound on the maximal algebraic error, dependent on the choice of basis. Here, we have used the fact that 
\[
\min_{\Vert \mathbf{b} \Vert_2=1}\Vert \mathbf{D}_\alpha\mathbf{b}\Vert_2=\sigma_{\min},
\]
where $\sigma_{\min}$ is the smallest singular value of $\mathbf{D}_\alpha.$ We thus choose $\mathbf{b} = \mathbf{v}_{\min},$ the right singular vector corresponding to $\sigma_{\min},$ as an approximate solution to the problem. The other singular vectors corresponding to larger singular values also give candidates for approximation that generally decrease in quality as the singular values increase \cite{dokken_2006}. It is important to note that the value of $\sigma_{\min}$ is dependent on the choice of basis $\bm\alpha.$

In this paper we use the following `normalization' to compare the approximation qualities of different polynomial bases:
\begin{equation}\label{eq:normalization}
\max_{t\in\Omega} \Vert \bm\alpha(t)\Vert_2=1.
\end{equation} 
It should be noted that bases with different scaling coefficients on the individual basis functions will produce different results. For example, the standard Legendre basis, where each basis function $(P_j)_{j=1}^L$ has the normalization $P_j(1) = 1$, will produce quantitatively different results to the Legendre basis normalized with respect to the $L_2$-inner product. This choice of scaling is somewhat arbitrary, but experience shows that small differences in the scaling result in small differences in the approximation. Thus, for the bases we study, standard choices will be made.

Given a choice of basis functions $\bm\alpha = (\alpha_i)_{i=1}^{L},$ we summarize the general approach of this section in the following algorithm:\newline

\begin{algorithm}
Input: a rational parametric curve $\mathbf{p}(t)$ of degree $n$ on the interval $[0,1],$ and a degree $m$ for the implicit polynomial:
\begin{enumerate}
\item For each basis function $(q_k)_{k=1}^M,$ compute the vector $\mathbf{d}_k = (d_{j,k})_{j=1}^{L}$ of coefficients such that $q_k(\mathbf{p}(t)) = \sum_{j=1}^L d_{j,k} \alpha_j(t),$
\item Construct a matrix $\mathbf{D}_\alpha = (\mathbf{d}_k)_{k=1}^M$ from the column vectors $\mathbf{d}_k,$
\item Perform an SVD $\mathbf{D}_\alpha=\mathbf{U}\bm{\Sigma}\mathbf{V}^T,$ and select $\mathbf{b}=\mathbf{v}_{\min},$ the right singular vector corresponding to the smallest singular value $\sigma_{\min}$ as an approximate solution.
\end{enumerate}
\end{algorithm}
This algorithm is known as the original method in the $\bm\alpha$-basis; however, for this paper, we will refer to it simply as the $\bm\alpha$-method for an arbitrary basis $\bm\alpha.$

\section{Weak approximate implicitization}\label{sec:weak}

Two approaches to approximate implicitization by continuous least squares minimization of the objective function were introduced simultaneously in 2001 in \cite{corless_2001,dokken_2001_2}, and further developed in \cite{dokken_2006}. These methods perform minimization in the weighted $L_2$-inner product:
\begin{equation}\label{eq:lsprob}
\min_{\Vert\mathbf{b}\Vert_2 = 1}\langle q\circ\mathbf{p}\ ,\ q\circ\mathbf{p}\rangle_w = \min_{\Vert\mathbf{b}\Vert_2 = 1} \int_{\Omega} w(t)q(\mathbf{p}(t))^2 \text{ d}t,
\end{equation}
where $w(t)$ is some weight function defined on the domain of approximation $\Omega.$

After choosing a basis for the implicit representation we obtain a linear algebra problem as before:
\begin{equation}\label{eq:minqf}
\min_{\Vert\mathbf{b}\Vert_2 = 1} \mathbf{b}^T\mathbf{M}_w\mathbf{b},
\end{equation}
where $\mathbf{M}_w = (m_{k,l})_{k=1,l=1}^{M,M}$ is given by 
\begin{equation}\label{eq:M}
m_{k,l} = \langle q_k\circ\mathbf{p}\ ,\ q_l\circ\mathbf{p} \rangle_w.
\end{equation}
This approach eliminates the need for a choice of basis, but a choice of weight function is necessary. The standard approach in \cite{corless_2001,dokken_2001_2} has been to take $w(t)\equiv 1.$ Later we will discuss the benefits of choosing the Chebyshev weight function on $[0,1],$ $w(t)=1/\sqrt{t(1-t)},$ instead. 

This problem, unlike the minimax problem, can be solved directly if the parametric components are integrable. We simply take $\mathbf{b} = \mathbf{v}_{\min},$ the eigenvector corresponding to the smallest eigenvalue of $\mathbf{M}_w.$ Since the matrix is symmetric, it has orthonormal eigenvectors. Similarly to the previous method, eigenvectors corresponding to larger eigenvalues give gradually degenerating approximations.

We summarize this algorithm, for a given weight function $w,$ as follows:\newline

\begin{algorithm}
Input: a parametric curve $\mathbf{p}(t)$ on the interval $[0,1],$ and a degree $m$ for the implicit polynomial:
\begin{enumerate}
\item Construct a matrix $\mathbf{M}_w$ by performing the integrals according to 
\[m_{k,l} = \langle q_k\circ\mathbf{p},q_l\circ\mathbf{p} \rangle_w\]
for $k,l=1,\ldots,M,$
\item Compute the eigendecomposition $\mathbf{M}_w = \mathbf{V}\bm{\Lambda}\mathbf{V}^T,$
\item Select $\mathbf{b}=\mathbf{v}_{\min},$ the eigenvector corresponding to the smallest eigenvalue $\lambda_{\min}.$
\end{enumerate}
\end{algorithm}

Algorithms following the procedure above are known under different names in the literature; weak approximate implicitization in \cite{dokken_2006}, numerical implicitization in \cite{corless_2001} and the eigenvalue method in \cite{emiris_2005}. For the rest of this paper we will call it the weak method. 

As previously mentioned, the methods of this section are suitable for either exact or approximate implicitization. They can be performed using either symbolic or numerical integration, however the former is generally only required when performing exact implicitizations in exact precision. For applications where floating point precision is sufficient, numerical quadrature rules provide a much faster alternative. The methods also have wide generality since they can be applied to any parametric forms with integrable components, not only rational parametric forms. There are, however, some significant disadvantages in choosing this method in practice. Firstly, due to the high degree of the integrand, the integrals can take a relatively long time to evaluate, even when numerical quadrature rules are used. Secondly, and more importantly, the condition numbers of the matrices $\mathbf{M}_w$ can be significantly larger than the condition numbers of the $\mathbf{D}_{\alpha}$ matrices from the previous section, leading to issues with numerical stability.

Since $\mathbf{M}_w$ is a symmetric positive (semi) definite matrix, it has a decomposition $\mathbf{M}_w = \mathbf{K}^T\mathbf{K},$ where the singular values of $\mathbf{K}$ are the square roots of the singular values of $\mathbf{M}_w.$ This decomposition is not unique; however, in the next section we show that it is possible to construct such a matrix directly, without first computing $\mathbf{M}_w.$ The condition number of $\mathbf{K}$ will be the square root of the condition number of $\mathbf{M}_w,$ hence we obtain the solution to the least squares problem in a more stable manner. Example \ref{ex:circle} demonstrates how the lack of stability in the weak method compares to the new method described in the following section.

\section{Approximate implicitization using orthonormal bases}\label{sec:orthog}

To make the connection between the original method and the weak method in the previous section, we consider the factorization (\ref{eq:fact}). We can then express $\mathbf{M}_w$ in terms of $\mathbf{D}_\alpha$ and a new matrix $\mathbf{A}$ \cite{dokken_2006}. That is, we get
\begin{equation}\label{eq:dad}
\int_{\Omega} w(t)q(\mathbf{p}(t))^2 \text{ d}s = \mathbf{b}^T \mathbf{D}_\alpha^T \mathbf{A} \mathbf{D}_\alpha \mathbf{b}
\end{equation}
where $\mathbf{A} = (a_{i,j})_{i=1,j=1}^{L,L}$ is given by $a_{i,j} = \langle \alpha_i,\alpha_j \rangle_w.$
Note that $\mathbf{A}$ is a Gramian matrix in the $\bm\alpha$-basis on the weighted $L^2$-inner product. This gives us a clue as to how to improve the weak method by the use of orthonormal bases. A polynomial basis $(\alpha_i)_{i=1}^L$ is said to be orthonormal with respect to the weighted $L^2$-inner product $\langle \cdot, \cdot \rangle_w,$ if 
\[
\langle \alpha_i,\alpha_j \rangle_w = \delta_{ij}, \quad \text{ for all } i,j=1,\ldots,L,
\]
with $\delta_{ij}$ denoting the Kronecker delta.

\begin{thm}\label{thm:DTD}
Let $\bm\alpha$ be a polynomial basis, orthonormal with respect to the given inner product $\langle\cdot,\cdot\rangle_w,$ and $\mathbf{D_{\alpha}}$ and $\mathbf{M}_w$ be defined by (\ref{eq:fact}) and (\ref{eq:M}) respectively. Then $\mathbf{M}_w = \mathbf{D}^T_{\alpha}\mathbf{D_{\alpha}}.$
\end{thm}

\begin{proof}
By (\ref{eq:dad}) and the definition of $\mathbf{M}_w,$ we have $\mathbf{M}_w = \mathbf{D}_\alpha^T \mathbf{A} \mathbf{D}_\alpha$ for any basis $\bm\alpha.$ But since $\bm\alpha$ is orthonormal with respect to the inner product ${\langle\cdot,\cdot\rangle_w},$ we have $a_{i,j} = \langle \alpha_i,\alpha_j \rangle_w = \delta_{i,j}$ (i.e., $\mathbf{A}=\mathbf{I}$). 
\end{proof}

We notice that the right singular vectors of $\mathbf{D}_\alpha$ are exactly the orthonormal eigenvectors of $\mathbf{D}_\alpha^T\mathbf{D_\alpha},$ and the singular values of $\mathbf{D}_\alpha$ are the non-negative square roots of the eigenvalues of $\mathbf{D}_\alpha^T\mathbf{D_\alpha}.$ Thus, with Theorem \ref{thm:DTD} in mind, we see that if the $\bm\alpha$-basis is orthonormal with respect to $w,$ then the results of the original method and the weak method are the same\footnote{Note that we do not specify that the orthonormal basis must be a polynomial basis, only a basis for the relevant space of functions. For example, curves defined by trigonometric polynomials can be treated in a similar way, since that basis is orthogonal.}. That is, $\mathbf{D}_{\alpha}$ is a candidate for the matrix $\mathbf{K}$ from the previous section. Such a matrix, as defined by (\ref{eq:fact}), can also be given elementwise by
\begin{equation}\label{eq:orthelement}
d_{j,k} = \int_{\Omega} w(t)\alpha_j(t)q_k(\mathbf{p}(t)) \text{ d}t.
\end{equation}
In practice, there often exist algorithms for computing coefficient expansions in orthogonal bases that are more efficient than using the elementwise definition. We will mention later how the Chebyshev and Legendre methods can utilize algorithms based on the fast Fourier transform (FFT) to generate the $\mathbf{D}_\alpha$ matrices.

We should note that the matrix $\mathbf{D}_\alpha$ is of dimension ${L\times M},$ compared with the ${M\times M}$ matrix $\mathbf{M}_w.$ For ${n\geq3},$ we have ${L\geq M},$ so finding the singular vectors of $\mathbf{D}_\alpha$ may be more computationally expensive than computing the eigenvectors of $\mathbf{M}_w.$ However, since the matrix construction is usually the dominant part of the algorithm, these differences do not affect the overall complexity of the algorithms. Moreover, the increase in accuracy more than justifies any small increase in computational complexity in part of the algorithm.

\subsection{Example}\label{ex:circle}

In order to compare the numerical stability of the two approaches to least squares minimization of the objective function, we turn to a familiar example; exact implicitization of a rational parametric circular arc, which is defined in projective space by 
\[
\mathbf{p}(t) = 
	\begin{pmatrix}
	p_1(t), p_2(t), h(t)	
	\end{pmatrix} = 
	\begin{pmatrix}
	{2t}, {1-t^2}, {1+t^2}
	\end{pmatrix}, \quad
	t\in[0,1].
\]
We perform the implicitization using the homogeneous Bernstein basis functions $q_k(u,v,w)\in\{u^2,2uv,2uw,v^2,2vw,w^2\},$ for $k=1,\ldots,6,$ and degree $m=2.$ If performed using exact arithmetic, the implicitization vector given by both the weak and orthonormal basis methods is 
\[
\mathbf{b} = (3^{-1/2},0,0,3^{-1/2},0,-3^{-1/2}).
\]
However, performing the algorithms in double precision, we obtain the results
\begin{eqnarray*}
\begin{split}
\mathbf{b}_\mathrm{weak} =  (0.577350269173099, -1.63\times10^{-11}, 2.19\times10^{-11}&,\\0.577350269178602, 1.87\times10^{-11},-0.5773&50269217176)
\end{split}\\
\begin{split}
\mathbf{b}_\mathrm{orth} = (0.577350269189627, 5.52\times10^{-16}, -8.47\times10^{-16}&,\\0.577350269189626, -4.16\times10^{-16}, -0.5773&50269189625)
\end{split}
\end{eqnarray*}
for the weak and the orthogonal basis methods respectively\footnote{The orthogonal basis method was implemented with Legendre polynomials, as described in Section \ref{ssec:jacobi}.}. Examining the respective relative errors (in the infinity norm)
\begin{eqnarray*}
\Vert \mathbf{b}_\mathrm{weak} -\mathbf{b} \Vert_\infty / \Vert \mathbf{b} \Vert_\infty & = & 4.77\times10^{-11}, \\
\Vert \mathbf{b}_\mathrm{orth} -\mathbf{b} \Vert_\infty / \Vert \mathbf{b} \Vert_\infty & = & 1.73\times10^{-15},
\end{eqnarray*}
we see that the orthogonal basis method preserves the accuracy much better than the weak method, with the former only preserving approximately 11 digits of accuracy. The orthogonal basis method preserves all but the last the digit of the implicitization, up to double precision. It should be noted that this is an example of implicitization with degree $m=2;$ as the degree is raised, such numerical errors can become more serious. The example in Section \ref{ssec:vis} shows how higher degree implicitizations using the weak method can give unpredictable results.

\section{Examples of the original approach with different bases}\label{sec:bases}

The previous sections justified why the original approach is preferable to the weak approach in cases where the expression $q(\mathbf{p}(t))$ can be expressed in terms of orthogonal functions. In this section we will look at examples of specific implementations using orthogonal polynomial bases. We will also consider alternative implementations using non-orthogonal bases, which produce approximations to the least squares solution. We state some simple propositions which unify the methods and also consider computational aspects of the algorithms.

\subsection{Jacobi polynomial bases}\label{ssec:jacobi}

The most commonly used orthogonal polynomial bases in approximation theory are the Legendre basis $(P_j)_{j=1}^L$ and the Chebyshev basis (of the first kind) $(T_j)_{j=1}^L.$ These are both special cases of Jacobi polynomials, which are orthogonal with respect to the weight functions defined by 
\[
w_{\alpha,\beta}(t) = t^\alpha (1-t)^\beta,
\]
on the interval $[0,1],$ and are defined for $\alpha,\beta>-1.$ The Legendre and Chebyshev cases are given by ${\alpha=\beta=0}$ and ${\alpha=\beta=-\frac{1}{2}}$ respectively. It is well known that the Chebyshev expansion has an efficient construction via use of the discrete cosine transform (DCT-I) which can be implemented via fast Fourier transform (FFT) \cite{battles_2004}. The Legendre expansion can also be constructed efficiently by transforming from the Chebyshev coefficients (in $O(n)$ operations for a degree $n$ polynomial \cite{alpert_1991}). The implementation of approximate implicitization exploiting the speed of the FFT will be the subject of another paper. Here we will look at the properties of the matrices $\mathbf{D}_P$ for the Jacobi bases, and the properties of the approximations. The following proposition is an analogue of Theorem 4.3 in \cite{dokken_2001}, for Jacobi polynomial bases.

\begin{prop}
Let $\mathbf{D}_P$ be the matrix defined by $(\ref{eq:fact})$ in a Jacobi polynomial basis ${(P_j(t))_{j=1}^L},$ for any ${\alpha,\beta>-1},$ normalized to ${\Vert P_j \Vert_{\infty}=1}$ for ${j=1,\ldots,L}.$ Then
\[
\sum_{k=1}^M d_{j,k} = 
\begin{cases}
1, & \text{if } j = 1, \\
0, & \text{if } j \geq 2. \\
\end{cases}
\]
\end{prop}

\begin{proof}
Since an orthogonal polynomial basis is degree ordered, one of the functions must be identically a non-zero constant, which, by the normalization condition is equal to 1. Consider the vector ${\mathbf{b} = (1,\dots,1)},$ which ensures that ${q(\mathbf{p}(t))=1}$ for all ${t\in [0,1]},$ by the partition of unity on the implicit basis ${(q_k)_{k=1}^M}.$ Clearly, only the basis function of degree zero can have a non-zero coefficient. By (\ref{eq:fact}), we have the expansion
\begin{eqnarray*}
q(\mathbf{p}(t)) & = & \sum_{j=1}^L P_j(t) \sum_{k=1}^M d_{j,k}. 
\end{eqnarray*}
But this is equal to 1 if and only if $\sum_{k=1}^M d_{1k} = 1$ and $\sum_{k=1}^M d_{j,k} = 0$ for $j\geq 2.$ 
\end{proof}

One property of Chebyshev expansions of a continuous function is that the error introduced by truncating the expansion is dominated by first term after the truncation, if the coefficients decay quickly enough \cite{gil_2007}. For curves that we wish to approximate (with relatively simple forms), the coefficients do tend to decay quickly, so the coefficients in the lower rows of the matrix tend to be dominated by those above them. 

The Chebyshev basis is well known for giving good approximations to minimax problems in approximation theory (see \cite{gil_2007} for an overview). This also seems to be the case for approximate implicitization, with the resulting error normally being close to equioscillating. In fact, experiments show that in almost all test cases, the number of roots in the error function given by the Chebyshev method is greater than or equal to the convergence rate, for the given implicit degree $m$ (see Table \ref{tab:conv_rates_curves}). Thus the Chebyshev method appears to give a `near best' approximation in the sense that the error normally oscillates a maximum number of times.

Our experience in the choice between Legendre and Chebyshev polynomials is that the difference in approximation quality is minor. Chebyshev expansions are slightly quicker to compute and require less programming effort than their Legendre counterparts \cite{alpert_1991}. In addition they tend to eliminate the spike in error at the end of the intervals that appears in the Legendre method. However, both algorithms provide efficient and numerically stable methods for (weighted) least squares approximation over the entire interval $\Omega.$

\subsection{Discrete approximate implicitization - the Lagrange basis}\label{ssec:lag}

One of the simplest and fastest implementations of approximate implicitization is to perform discrete least squares approximation of points sampled on the parametric manifold, similar to the methods in \cite{bajaj_1993,pratt_1987}. In our setting, this can be implemented as the original method but with the $\bm\alpha$-basis chosen to be the Lagrange basis at the given nodes. The matrix $\mathbf{D}_{\mathrm{L},L}$ defined by $(\ref{eq:prob})$ in the Lagrange basis of degree $L-1$ can be given elementwise by 
\begin{equation}\label{eq:lagmat}
d_{j,k} = q_k(\mathbf{p}(t_j)), \quad j=1,\ldots,L, \text{ and } k=1,\ldots,M,
\end{equation}
where $t_j\in\Omega$ are nodes in the parameter domain. 

The result of approximate implicitization in the Lagrange basis depends both on the number of points sampled and the density of the point distribution in the parameter domain. Since Lagrange polynomials are neither orthogonal nor degree ordered, they do not solve a least squares problem of type (\ref{eq:lsprob}). However, we can form a direct relation between the discrete and continuous least squares problems, as follows: 

\begin{prop}\label{prop:lagrange_convergence}
Let $\mathbf{p}(t)$ be a planar parametric curve with bounded, piecewise continuous components on the interval $[0,1]$ and let $\mathbf{D}_{\mathrm{L},L}$ be the matrix defined by (\ref{eq:lagmat}) at uniform nodes. Then the $(k,l)$ element of ${\mathbf{D}_{\mathrm{L},L}}^T\mathbf{D}_{\mathrm{L},L}$ converges to a constant multiple of the $(k,l)$ element of $\mathbf{M}_1,$ defined by (\ref{eq:M}), as the number of samples $L$ is increased.
\end{prop}
\begin{proof}
Since each of the parametric components of the curve is bounded and piecewise continuous on the interval and $q_k$ is a polynomial, we know that $q_k\circ\mathbf{p}$ is bounded and piecewise continuous for each $k.$ Let $h_L := 1/(L-1).$ Then, for uniform samples $(t_j)_{j=1}^L,$ where $t_j = h_L(j-1)$ in the parameter domain we see that
\begin{eqnarray*}
\lim_{L\rightarrow\infty}h_L ({\mathbf{D}_{\mathrm{L},L}}^T\mathbf{D}_{\mathrm{L},L})_{k,l} & = & \lim_{L\rightarrow\infty}\sum_{j=1}^L h_L q_k(\mathbf{p}(t_j))q_l(\mathbf{p}(t_j)), \\
& = & \int_{\Omega} q_k(\mathbf{p}(t))q_l(\mathbf{p}(t)) \text{ d}t, \\
& = & (\mathbf{M}_1)_{k,l},
\end{eqnarray*}
for $k,l=1,\ldots,M.$
\end{proof}
Sampling more points gives ever closer approximations of the true least squares approximation (for any given $L,$ the constant $h_L$ is absorbed into the singular values of the matrix and does not affect the singular vectors). However, as we have seen, the Legendre method can solve the (unweighted) least squares problem exactly, without excessive sampling and in a way that best preserves the numerical precision. Although the point sampling method does tend to give good approximations when the number of samples $L$ is large enough, it is a relatively small increase in computational complexity and programming effort to use Legendre expansions instead. The main strength of the Lagrange method lies in its simplicity: it is easy to implement, computationally inexpensive and highly parallelizable. 
 
Alternative choices of nodes are also interesting to investigate. Using the inequality (\ref{eq:ineq}), we can introduce the bound
\begin{eqnarray*}
\min_{\Vert\mathbf{b}\Vert_2=1}\max_{t\in\Omega} |q(\mathbf{p}(t))| & \leq & \Vert \bm\alpha(t) \Vert_2 \ \sigma_{\min}, \\
& \leq & \Lambda(\bm\alpha) \ \sigma_{\min}, \\
\end{eqnarray*}
where $\Lambda(\bm\alpha)$ is the Lebesgue constant from interpolation theory defined by $\Lambda(\bm\alpha) = \max_{t\in\Omega} \Vert \bm\alpha(t) \Vert_1.$ Thus we may expect to obtain better results from point distributions with smaller Lebesgue constants. In particular, we will see in Section \ref{ssec:bern} that by using the Bernstein basis we achieve a Lebesgue constant of $1.$ However, it is possible that with smaller Lebesgue constants come larger singular values. Thus it is important to balance between minimizing the Lebesgue constant and the singular value in order to obtain the best bound for the algebraic error.

A point distribution of particular interest is that of the Chebyshev points. On $[0,1]$ this is defined as:
\begin{equation}\label{eq:cheb_points}
t_j = \frac{1}{2}\left(1-\cos\left(\frac{j\pi-\pi}{L-1}\right)\right), \quad j=1,\ldots,L.
\end{equation}
Conversion from the Lagrange basis at Chebyshev points to the Chebyshev basis can be performed by a DCT-I of the Lagrange coefficients \cite{gil_2007}. Implementing the algorithm in this way gives a fast procedure for the Chebyshev method of Section \ref{ssec:jacobi}. Proposition \ref{prop:lagrange_convergence} can be extended to show that sampling at an increasing number of Chebyshev points causes the solution to converge to that of the weak method with the Chebyshev weight function (see Proposition \ref{prop:cheb_lagrange_convergence}). However, interpolation in these points is known from approximation theory to give very good approximation properties of its own \cite{battles_2004}. One reason for this is the small Lebesgue constants associated with Chebyshev points. 

\begin{prop}\label{prop:cheb_lagrange_convergence}
Let $\mathbf{p}(t)$ be a planar parametric curve with bounded, piecewise continuous components on the interval $[0,1]$ and let $\mathbf{D}_{\mathrm{L},L}$ be the matrix defined by (\ref{eq:lagmat}) at Chebyshev points given by (\ref{eq:cheb_points}). Then the $(k,l)$ element of ${\mathbf{D}_{\mathrm{L},L}}^T\mathbf{D}_{\mathrm{L},L}$ converges to a constant multiple of the $(k,l)$ element of $\mathbf{M}_w,$ defined by (\ref{eq:M}) with the weight function $w(t) = 1/\sqrt{t(1-t)},$ as the number of samples $L$ is increased.
\end{prop}

\begin{proof}
We use the change of variable $t(\theta) = (1-\cos(\pi\theta))/2,$ and note first that 
\[
\sqrt{t(\theta)(1-t(\theta))} = \sin(\pi\theta)/2, \quad \text{ and } \ \ \frac{\text{d}t}{\text{d}\theta} = \pi\sin(\pi\theta)/2.
\]
As in Proposition \ref{prop:lagrange_convergence}, $q_k\circ\mathbf{p}$ is bounded and piecewise continuous for each $k.$ In order to simplify the notation, let $f_k=q_k\circ\mathbf{p}$ for each $k,$ and let $h_L = 1/(L-1).$ Then, for uniform samples $\theta_j$ in the domain $[0,1]$ (which correspond to the Chebyshev points $t_j=t(\theta_j)$ in the parameter domain), we see that
\begin{eqnarray*}
\lim_{L\rightarrow\infty}h_L ({\mathbf{D}_{\mathrm{L},L}}^T\mathbf{D}_{\mathrm{L},L})_{k,l} & = & \lim_{L\rightarrow\infty}\sum_{j=1}^L h_L f_k(t(\theta_j))f_l(t(\theta_j)), \\
& = & \int_{\Omega} f_k(t(\theta))f_l(t(\theta)) \text{ d}\theta, \\
& = & \int_{\Omega} \frac{f_k(t(\theta))f_l(t(\theta))\sin(\pi\theta)}{2\sqrt{t(\theta)(1-t(\theta))}} \text{ d}\theta, \\
& = & \frac{\pi}{2}\int_{\Omega} \frac{f_k(t)f_l(t)}{\sqrt{t(1-t)}} \text{ d}t, \\
& = & \frac{\pi}{2}(\mathbf{M}_w)_{k,l},
\end{eqnarray*}
for $k,l=1,\ldots,M.$ Note that $t(\Omega) = \Omega,$ so the integral is taken over the same domain even after change of variables. 
\end{proof}

\subsection{Bernstein polynomial basis}\label{ssec:bern}

The approach in \cite{dokken_1997} was to choose $\bm\alpha$ to be a non-negative partition of unity basis such as the Bernstein basis (i.e., $\sum_i \alpha_i(t)\equiv 1$ and $0\leq\alpha_i(t)\leq 1$). This ensures that $\Vert \bm\alpha(t)\Vert_2\leq \Lambda(\bm\alpha) = 1$ for all $t\in\Omega$ and so the smallest singular value gives an upper bound for the error 
\[
\min_{\Vert\mathbf{b}\Vert_2=1}\max_{t\in\Omega}|q(\mathbf{p}(t))| \leq \sigma_{\min}. 
\]
Approximation in the Bernstein basis $(B_j)_{j=1}^L,$ has the advantage that it is easily generalizable to both tensor-product and simplex domains in higher dimensions \cite{barrowclough_2010}. If the parametric curves are given in spline or B\'ezier form, it is natural to use the Bernstein coefficients, since there exist numerically stable algorithms for computing the compositions, without resorting to sampling \cite{derose_1993}. Despite the slightly less favourable approximation qualities of the Bernstein basis (see Figure \ref{fig:maxerror}), this method performs sufficiently well to be integrated into CAGD systems that are based on B\'ezier or spline curves and surfaces. It also appears to be the most stable method if the degree $m$ is chosen high enough for an exact implicit representation (see Section \ref{sec:comparison}).

The Bernstein method is closely related to both the Lagrange and Legendre methods seen previously. It is in fact easy to see that the $\mathbf{D}_{\mathrm{B},L}$ matrix in the Bernstein basis of degree $L-1,$ converges asymptotically to the $\mathbf{D}_{\mathrm{L},L}$ matrix in the Lagrange basis of degree $L-1$ at uniform nodes, as the degree is raised. 

\begin{prop}\label{prop:bernstein_convergence}
Let $\mathbf{D}_{\mathrm{B},L}$ be the matrix defined by (\ref{eq:prob}) in the degree $L-1$ Bernstein polynomial basis. Then the $(j,k)$ element of $\mathbf{D}_{\mathrm{B},L}$ converges to the $(j,k)$ element of $\mathbf{D}_{\mathrm{L},L},$ defined by (\ref{eq:prob}) in the Lagrange polynomial basis at uniform nodes.
\end{prop}
\begin{proof}
It is well known that  the Bernstein coefficients of a polynomial tend to the values of the polynomial as the degree is raised, as follows \cite{floater_2000}:
\[
\lim_{L\rightarrow\infty} (\mathbf{D}_{\mathrm{B},L})_{j,k} = q_k\left(\mathbf{p}\left(t_j\right)\right), \text{ for all } j=1,\ldots,L,
\]
where $t_j = (j-1)/(L-1).$ But the elements on the right-hand side are simply the elements of $\mathbf{D}_{\mathrm{L},L}$ in the Lagrange basis at uniform nodes. 
\end{proof}
We can thus deduce the following convergence property of the Bernstein method as an immediate consequence of Propositions \ref{prop:lagrange_convergence} and \ref{prop:bernstein_convergence}:
\begin{cor}\label{cor:bern_con}
Let $\mathbf{D}_{\mathrm{B},L}$ be the matrix defined by (\ref{eq:prob}) in the Bernstein polynomial basis of degree $L-1.$ Then the $(k,l)$ element of ${\mathbf{D}_{\mathrm{B},L}}^T\mathbf{D}_{\mathrm{B},L}$ converges to a constant multiple of the $(k,l)$ element of $\mathbf{M}_1,$ defined by (\ref{eq:M}), as the degree is raised.
\end{cor}

\subsection{Exact implicitization using linear algebra}\label{ssec:exact}

As mentioned previously, when the degree $m$ is chosen high enough to give an exact implicit representation and the algorithms are implemented in exact precision, all the methods can give exact results. The choice of basis in the exact case is irrelevant to the resulting polynomial and only affects only the implementation complexity and computational speed. For example, the elements of the matrix using the Lagrange method can be generated by choosing rational nodes represented in exact arithmetic. As with the floating point implementation, the matrix can be built very quickly in parallel, but rather than using SVD we can exploit algorithms for finding the kernel of a matrix. A similar method for exact implicitization is described in \cite{wang_2004}. However, the matrix there is expanded in the monomial basis, which leads to computationally expensive expansions for high degrees. It is noted that it is possible to reduce the complexity of that method in certain cases by exploiting the special structures of the algorithm and sparsity in the resulting matrix. In general, sparsity is not a feature of the Lagrange method, however, the matrices can be built more efficiently. The Bernstein method can also be implemented in exact arithmetic, however, similarly to the method in \cite{wang_2004}, it suffers from issues with computationally expensive expansions. Although it is possible to implement the Chebyshev and Legendre methods in exact precision using the elementwise definition (\ref{eq:orthelement}), the evaluation of such integrals will be slow. The fast algorithms for generating the matrices using FFT are not suitable for an exact implementation. 

When using the original method for approximate implicitization, we represent the error function $q\circ\mathbf{p}$ in a basis of degree $mn.$ In the Lagrange basis we thus choose the number of nodes $L,$ to be one more than the degree of the basis ${L=mn+1}.$ This is shown to be the smallest number of samples one can take in order to guarantee an exact implicitization method in the following proposition:

\begin{prop}\label{prop:samps}
Suppose we are given a non-degenerate rational parametric planar curve $\mathbf{p}(t)$ of degree $n$ (i.e., the degree of the algebraic representation is $n$). Then the number of unique samples required to guarantee an exact implicitization by the Lagrange method is given by $K=n^2+1.$
\end{prop}
\begin{proof}
Since the implicitization is exact, we know that there exists a unique polynomial $q$ of degree $n$ with coefficients $\mathbf{b}$ such that $q\circ\mathbf{p}\equiv0.$ By the theory described in Section \ref{sec:orig}, we can write 
\[
q(\mathbf{p}(t)) = \sum_{j=1}^K \alpha_j(t) \sum_{k=1}^M d_{j,k} b_k \equiv 0,
\]
where $(\alpha_j)_{j=1}^K$ is a basis for polynomials of degree $K-1,$ and $K=n^2+1.$ Since the polynomials $(\alpha_j)_{j=1}^K$ are linearly independent, we have 
\[
\sum_{k=1}^M d_{j,k} b_k = 0,
\]
for $j=1,\ldots,K,$ and since $\mathbf{D}_\alpha\neq0,$ the vector $\mathbf{b}$ must lie in the null space of $\mathbf{D}_\alpha$. This shows that any basis of degree $K-1$ can be used for exact implicitization. In the univariate case, Lagrange polynomials defined by $K$ points form a basis for polynomials of degree up to $K-1,$ if and only if all the $K$ points are unique. Thus, choosing $K$ unique points in the parameter domain is sufficient to guarantee an exact implicitization. 

To see that choosing fewer than $K$ points is insufficient, we consider parameter values corresponding to double points on the curve. Let $K_1=\frac{1}{2}(n+1)(n+2)-1$ denote the minimum number of points in general position on a curve of degree $n$ required to define the curve. Let the total number of possible double points on a rational curve of degree $n$ be given by ${K_2=\frac{1}{2}(n-2)(n-1)}$ \cite{milnor_1968}. Then up to $K_2$ points on the curve can correspond to more than one parameter value. Thus the minimum number of samples guaranteeing a unique exact implicitization is given by 
\[
K_1+K_2 = K.
\]
\end{proof}

When searching for exact implicitizations, we generally want the implicit polynomial of \emph{lowest degree} that contains the parametric curve. Since the normalized coefficient vector $\mathbf{b},$ given by an exact implicitization of lowest degree is unique, the kernel of the matrix would be expected to be 1-dimensional. A kernel of dimension higher than one indicates that the implicit polynomial defined by any vector in the kernel is reducible, and thus the degree $m$ can be reduced.

\section{Comparison of the algorithms}\label{sec:comparison}

So far, we have presented several approaches to exact and approximate implicitization using linear algebra. The approaches exhibit different qualities in terms of approximation, conditioning and computational complexity. The intention of this section is to provide a comparison of the algorithms.

\subsection{Algebraic error comparisons}

\begin{figure}
\center
\includegraphics[scale=1.0]{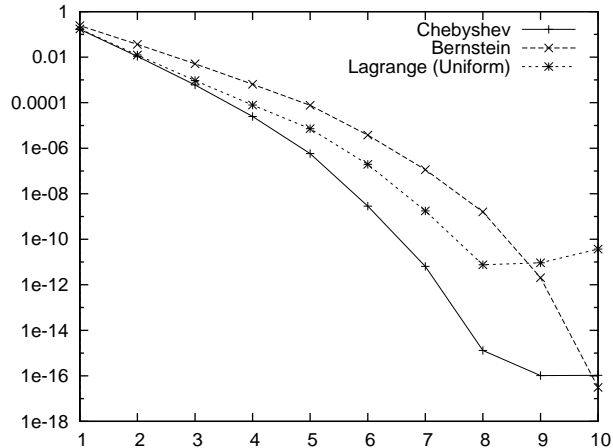}
\caption{The average uniform algebraic error, $\max{|q(\mathbf{p}(t))|},$ of the approximate implicitization of 100 B\'ezier curves of degree 10, with random control points in the triangle defined by $(1,0), (0,0)$ and $(0,1),$ with implicit approximation degrees between 1 and 10.}\label{fig:maxerror}
\end{figure}

Figure \ref{fig:maxerror} plots the average uniform algebraic error, $\max_{t\in[0,1]}|q(\mathbf{p}(t))|,$ of the approximations of 100 B\'ezier curves of degree 10, for algebraic degrees $m=1,\ldots,10.$ We use a barycentric coordinate system defined on the triangle $(1,0), (0,0)$ and $(0,1)$ for the implicit representation, and choose random control points within this region to define the B\'ezier curves. We compare the performance of the Lagrange basis (at uniform nodes)\footnote{Any reference to the Lagrange basis throughout this section will be assumed to be at uniformly spaced nodes.}, the Bernstein basis and the Chebyshev basis. As the results in the Chebyshev and Legendre bases are very similar, we include only the Chebyshev basis. The monomial basis, transformed to the interval $[0,1],$ was also considered but not included due to its vastly inferior approximation qualities. All the algorithms were performed in double precision.

For each degree up to the exact degree, $m=10,$ the Chebyshev basis gave the best uniform minimization of the objective function $q\circ\mathbf{p}.$ The maximum error for degree $m$ in the Chebyshev basis was, in general, approximately equivalent to the maximal error in the Lagrange basis of degree $m+1$ and the Bernstein basis of degree $m+2.$ For the Chebyshev basis, the maximal errors level out to roughly double-precision accuracy at degree eight, whereas for the Bernstein basis, the required degree for machine precision was 10. Although the Lagrange basis performs better than the Bernstein basis for low degrees, the higher degree approximations are distorted to the extent that the exact implicitization, at degree $10,$ loses several orders of magnitude in accuracy. This appears to be due to the large deviation in the extrema of Lagrange basis functions at uniform nodes, which are associated with large Lebesgue constants. An example of such an error distribution is pictured in Figure \ref{fig:lagerr}. The spike in error can be reduced by sampling more points, as the error converges to the weak approximation (c.f., Proposition \ref{prop:lagrange_convergence}), or by choosing point distributions with smaller Lebesgue coefficients. 

It should be noted that for the Bernstein and Lagrange methods, the maximum of the algebraic error normally occurs at the end points of the interval, and is normally much higher than the average error across the interval (see Figure \ref{fig:err}). Moreover, the error away from the ends of the interval can sometimes be smaller in these bases than in the Chebyshev basis. Hence, they generally perform better than Figure \ref{fig:maxerror} may suggest, in relation to the Chebyshev basis. The Chebyshev basis, however, tends to make the error roughly equioscillating throughout the interval. In addition, topological constraints imposed by approximating with lower degrees than necessary mean that even when the algebraic error is small, the geometric error can be somewhat different, especially for curves with many singularities.

\subsection{A visual comparison of the methods}\label{ssec:vis}

\begin{figure}
\begin{center}
\begin{tabular}{| c |}
\hline
\includegraphics[scale=0.21]{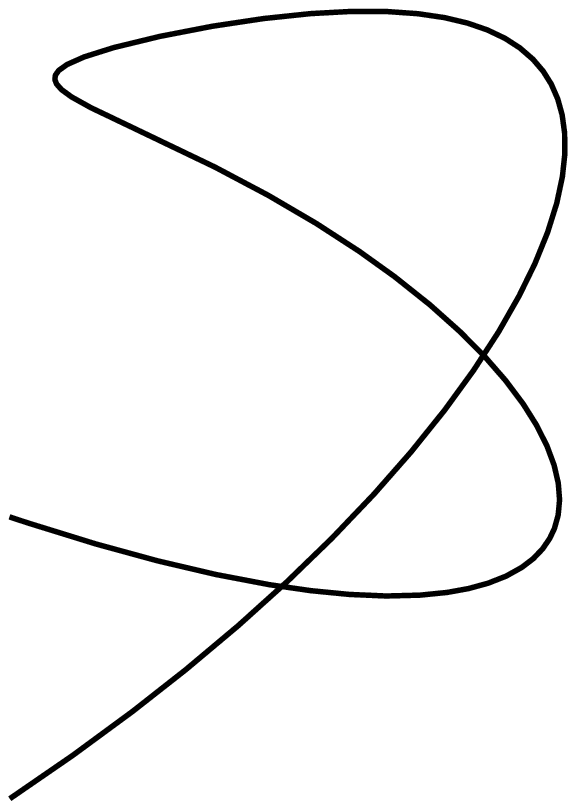} \\ \hline
\end{tabular}\caption{A parametric B\'ezier curve of degree seven, with control points $(\frac{1}{5},\frac{1}{10}),$ $(\frac{1}{2},\frac{3}{10}),$ $(\frac{1}{2},\frac{1}{2}),$ $(\frac{3}{10},\frac{1}{2}),$ $(0,0),$ $(0,\frac{4}{5}),$ $(\frac{4}{5},0)$ and $(\frac{1}{5},\frac{1}{5}).$ Implicit approximations of this curve appear in Figure \ref{fig:implicitizations}.}\label{fig:param}
\vspace*{1cm}
\begin{tabular}{ | l | c | c | c | c | }
\hline
 & Monomial & Bernstein & Lagrange & Chebyshev \\ \hline
$m=4$ & 
\includegraphics[scale=0.21]{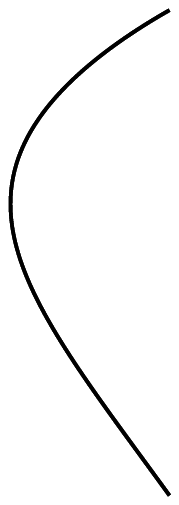} & 
\includegraphics[scale=0.21]{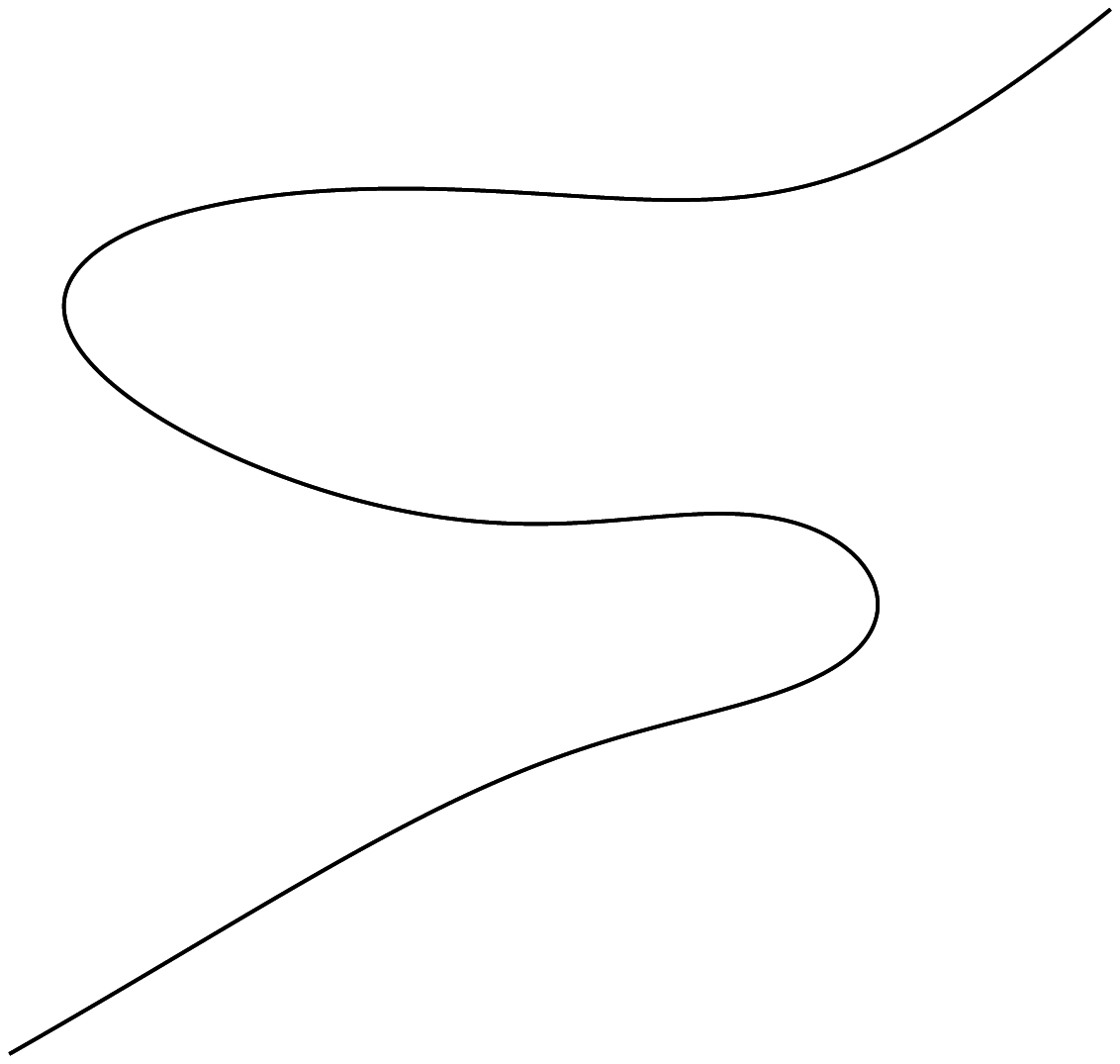} & 
\includegraphics[scale=0.21]{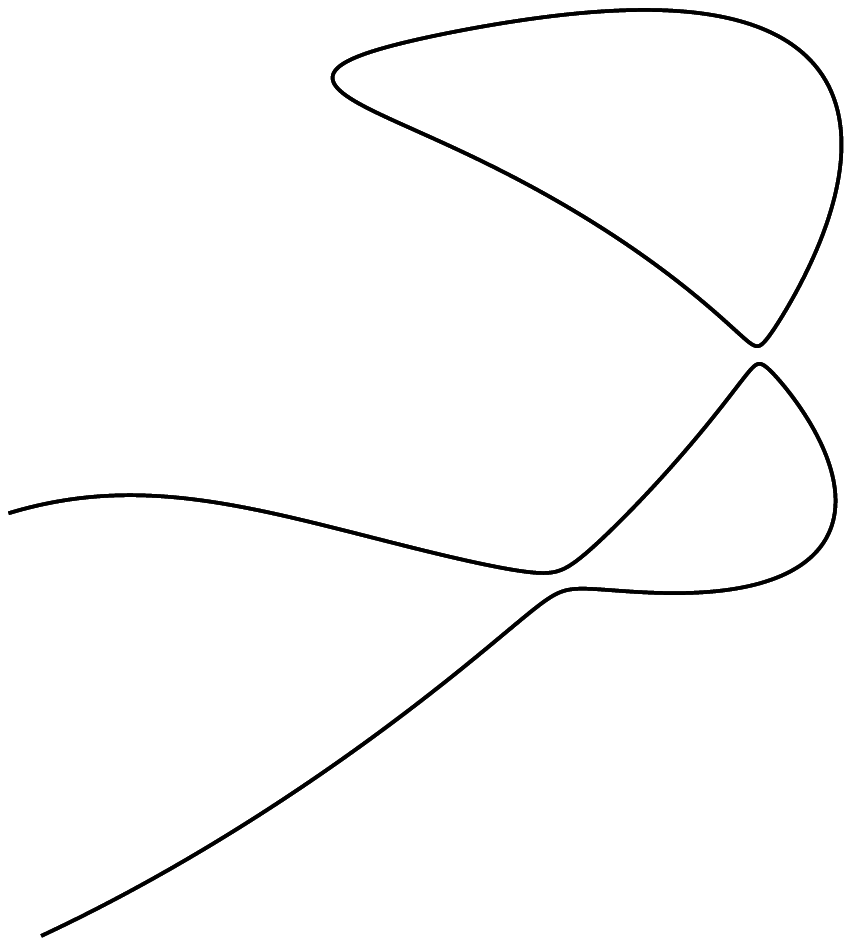} & 
\includegraphics[scale=0.21]{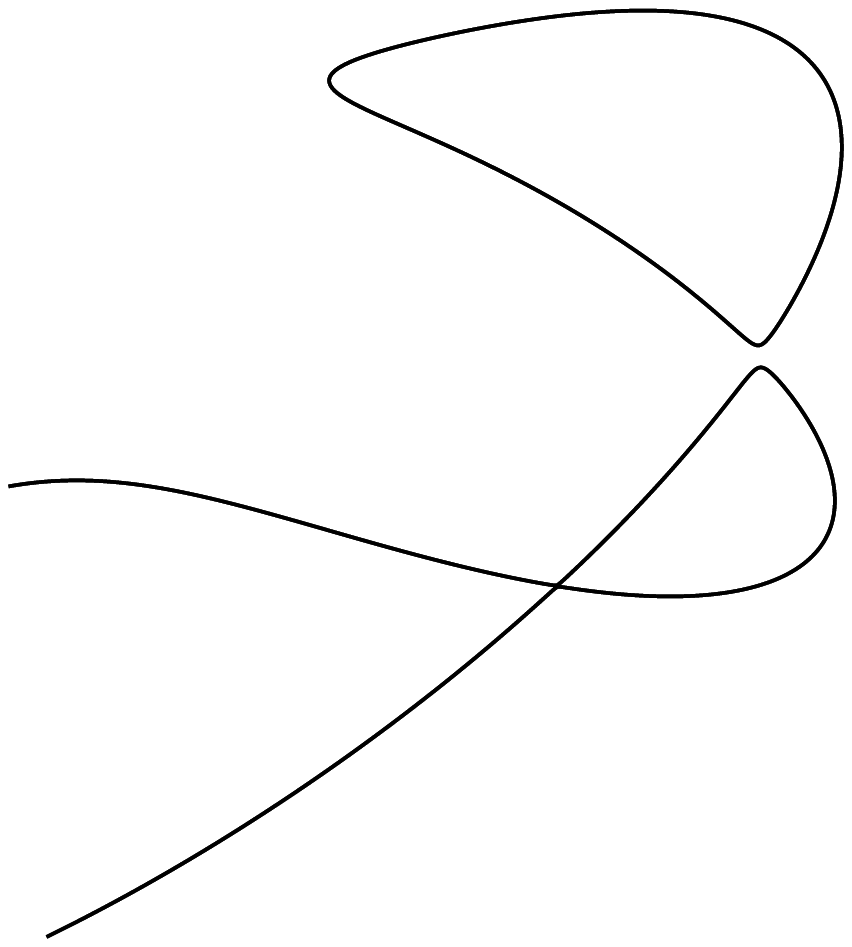} \\ \hline
$m=5$ & 
\includegraphics[scale=0.21]{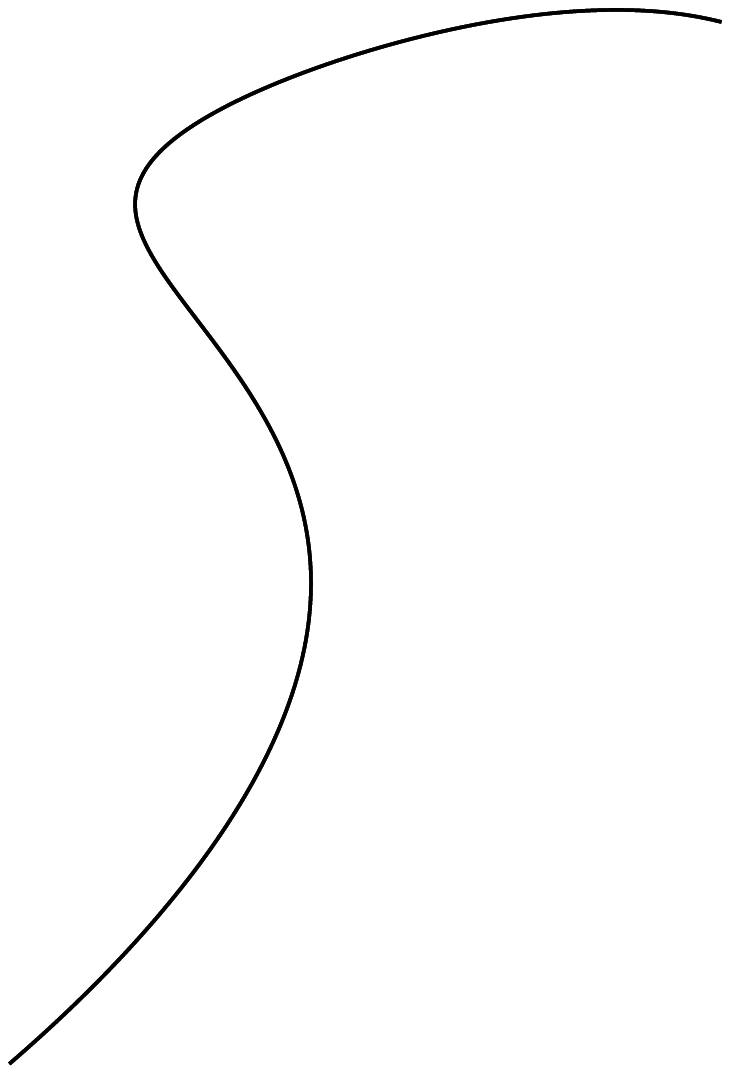} & 
\includegraphics[scale=0.21]{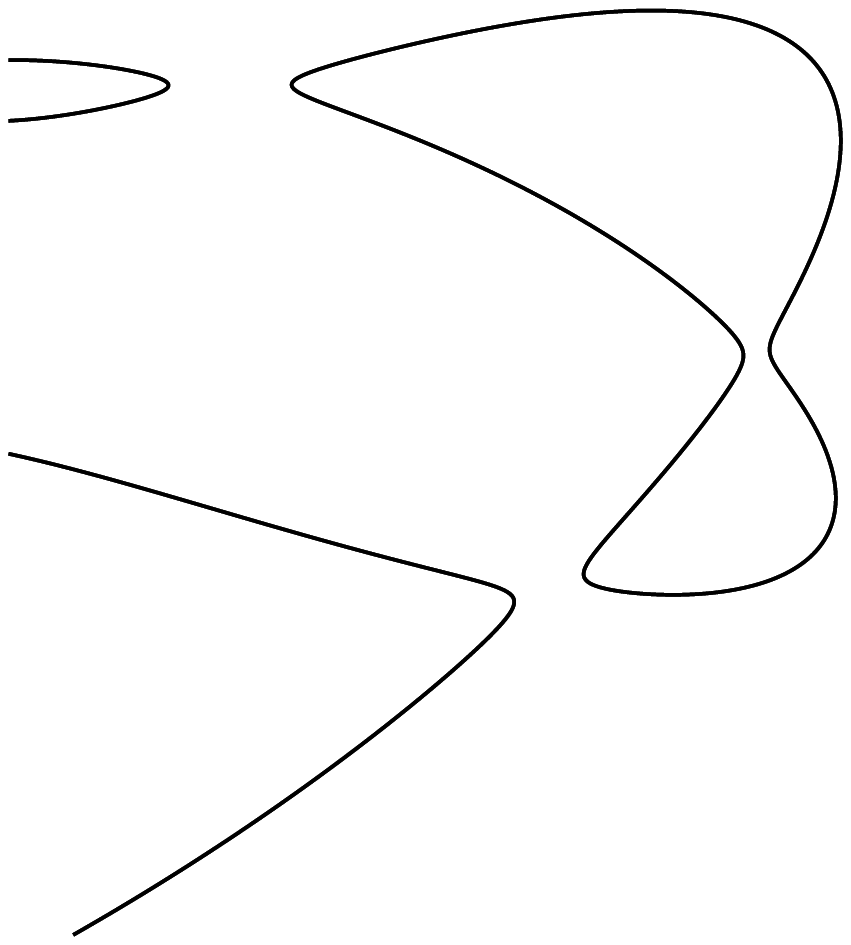} & 
\includegraphics[scale=0.21]{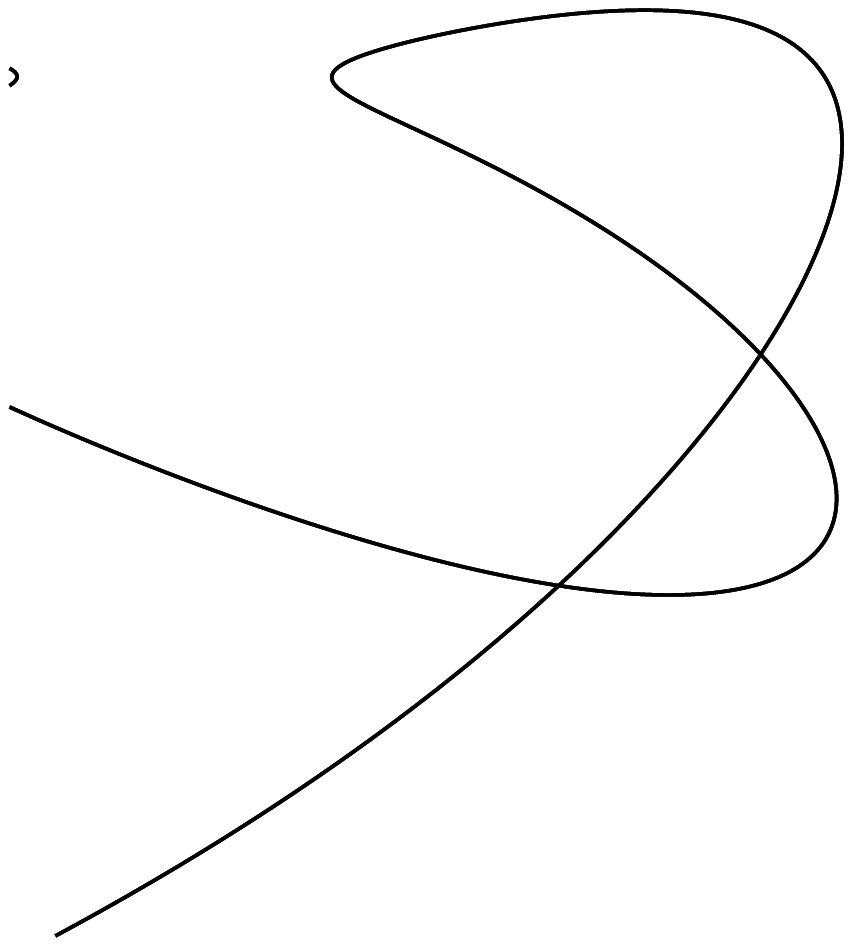} & 
\includegraphics[scale=0.21]{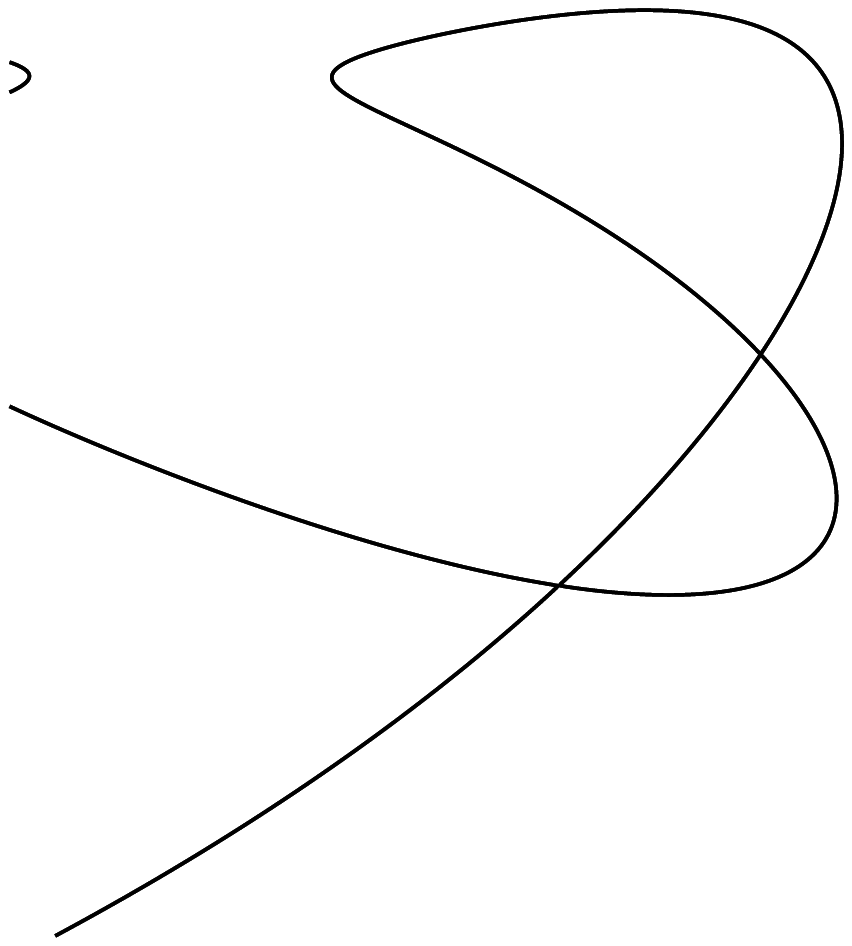} \\ \hline
$m=6$ &
\includegraphics[scale=0.21]{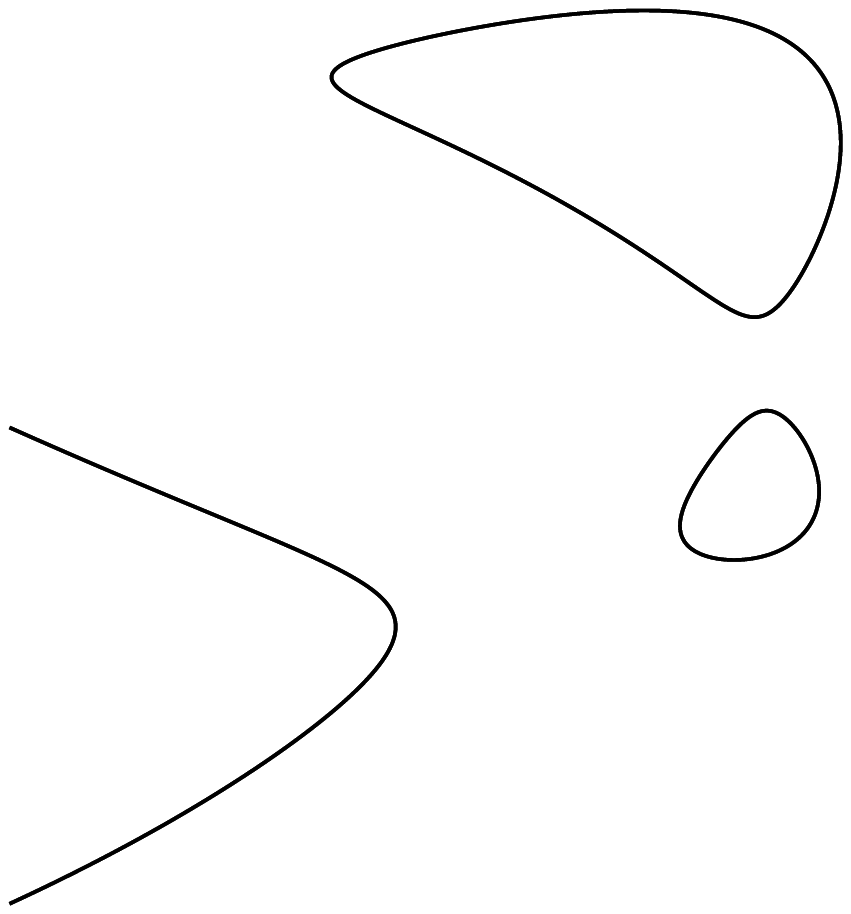} & 
\includegraphics[scale=0.21]{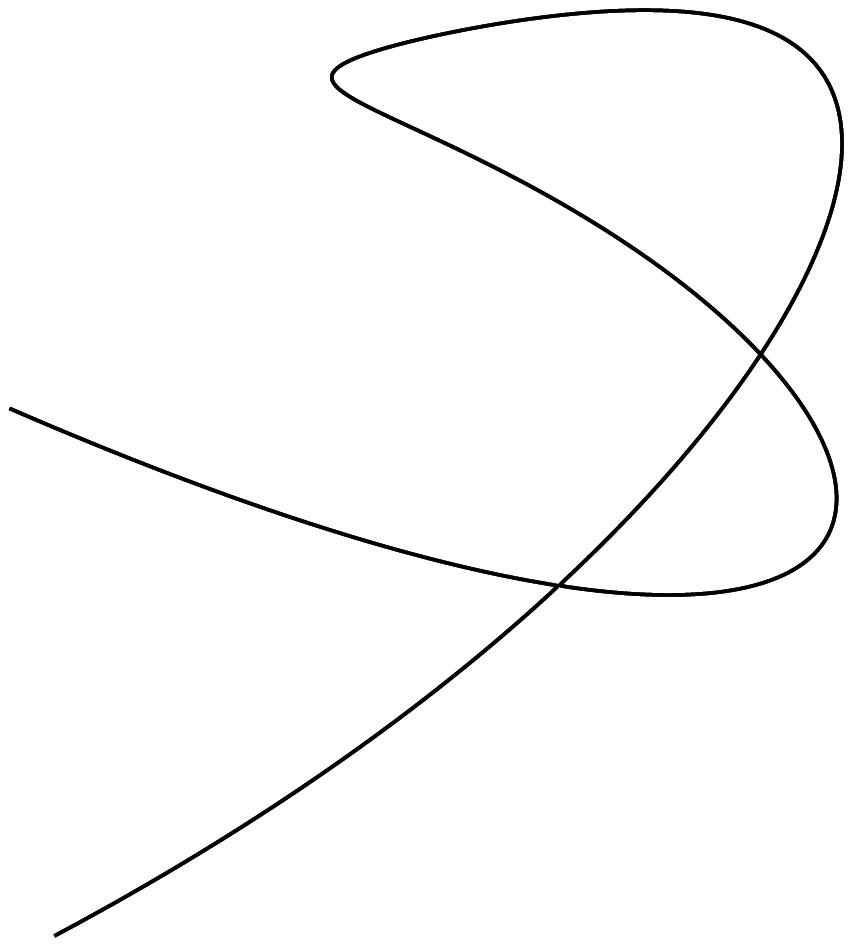} & 
\includegraphics[scale=0.21]{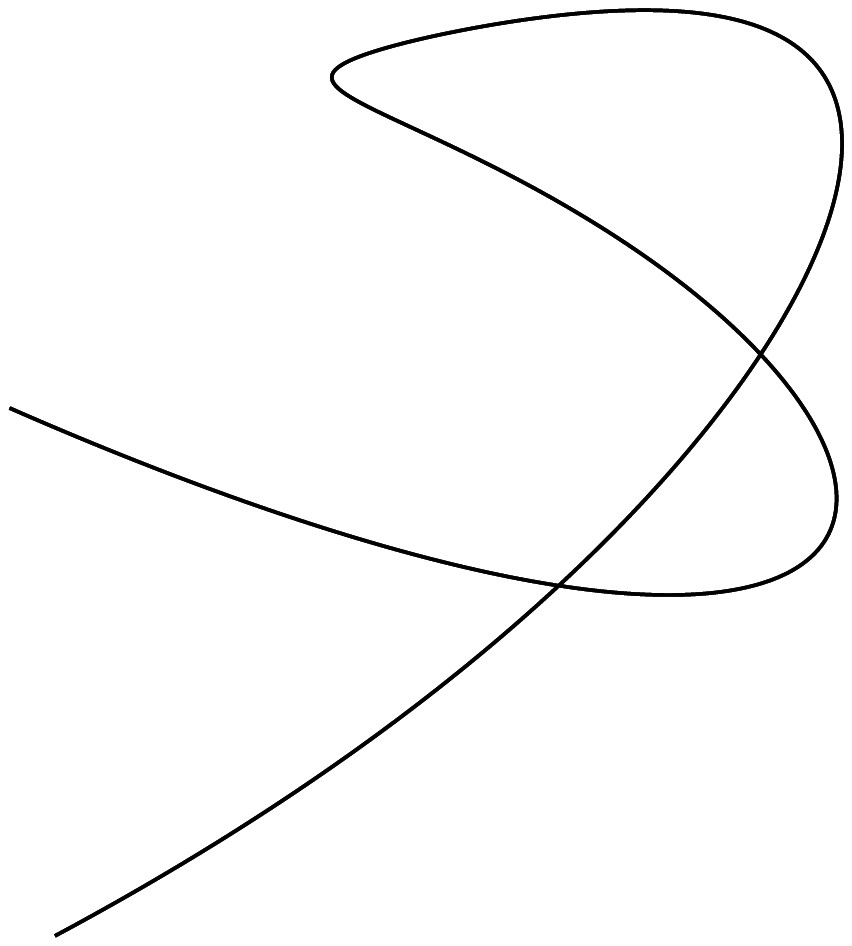} & 
\includegraphics[scale=0.21]{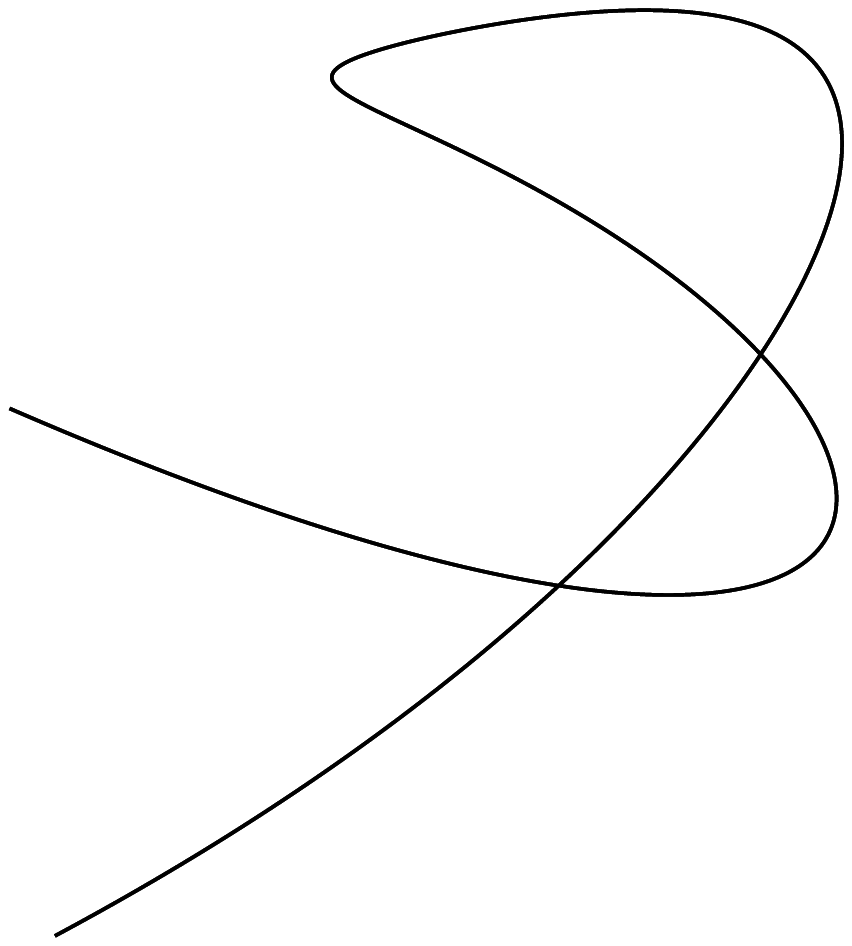} \\ \hline
$m=7$ &
\includegraphics[scale=0.21]{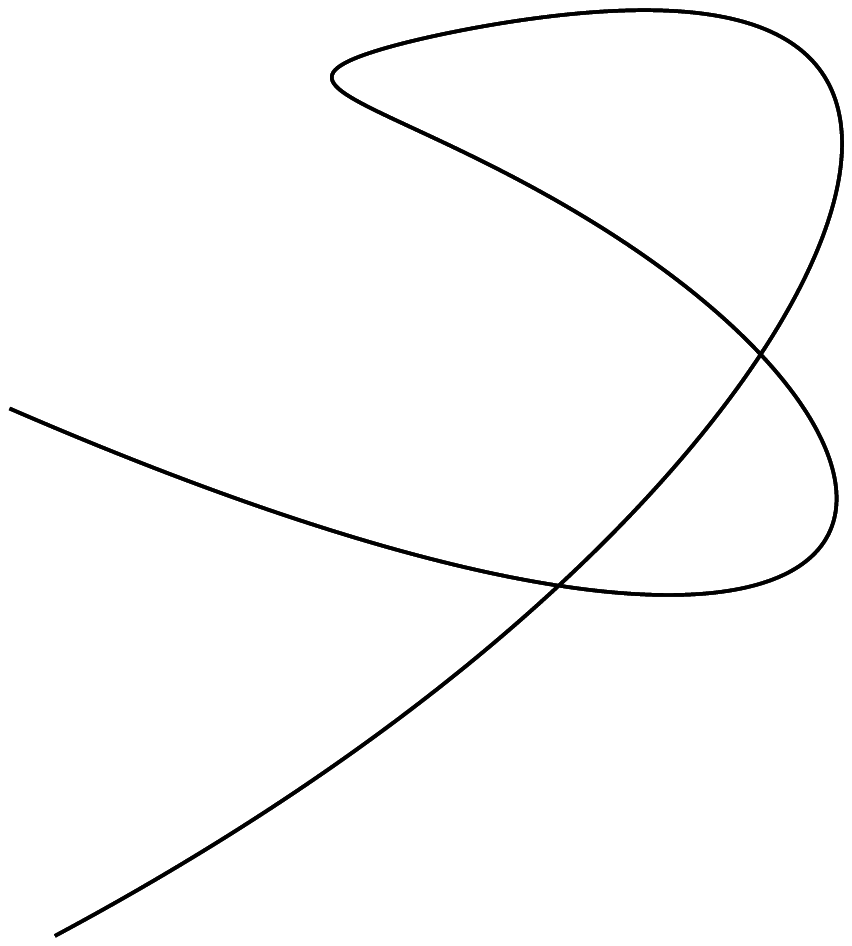} & 
\includegraphics[scale=0.21]{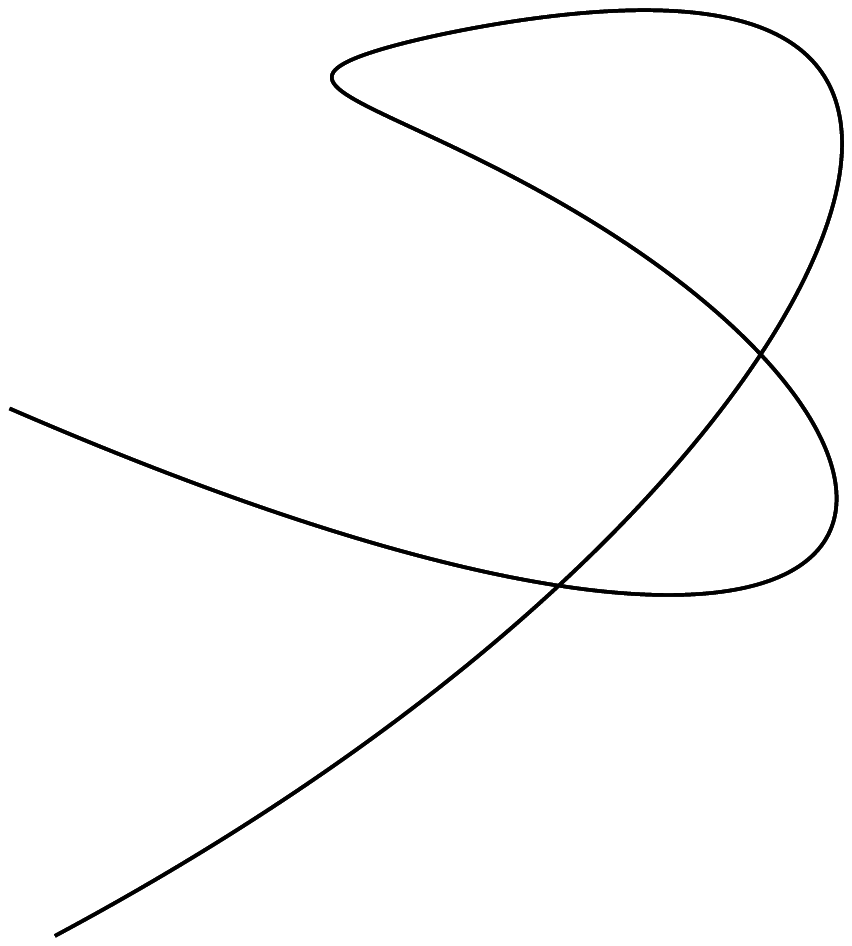} & 
\includegraphics[scale=0.21]{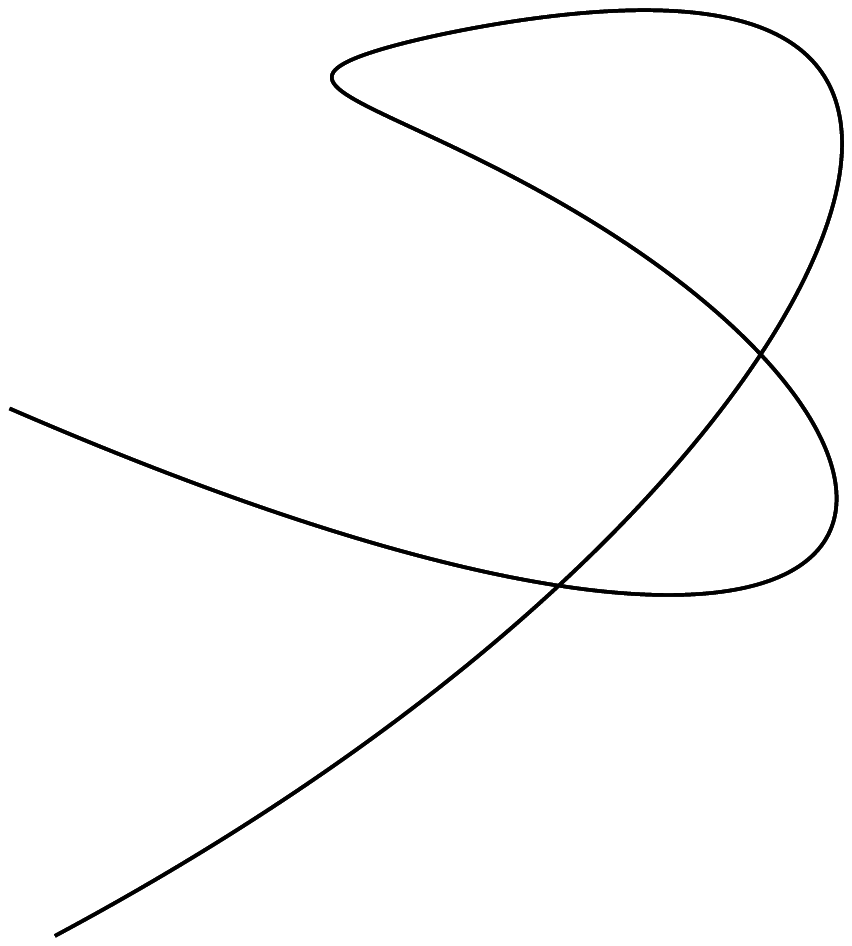} & 
\includegraphics[scale=0.21]{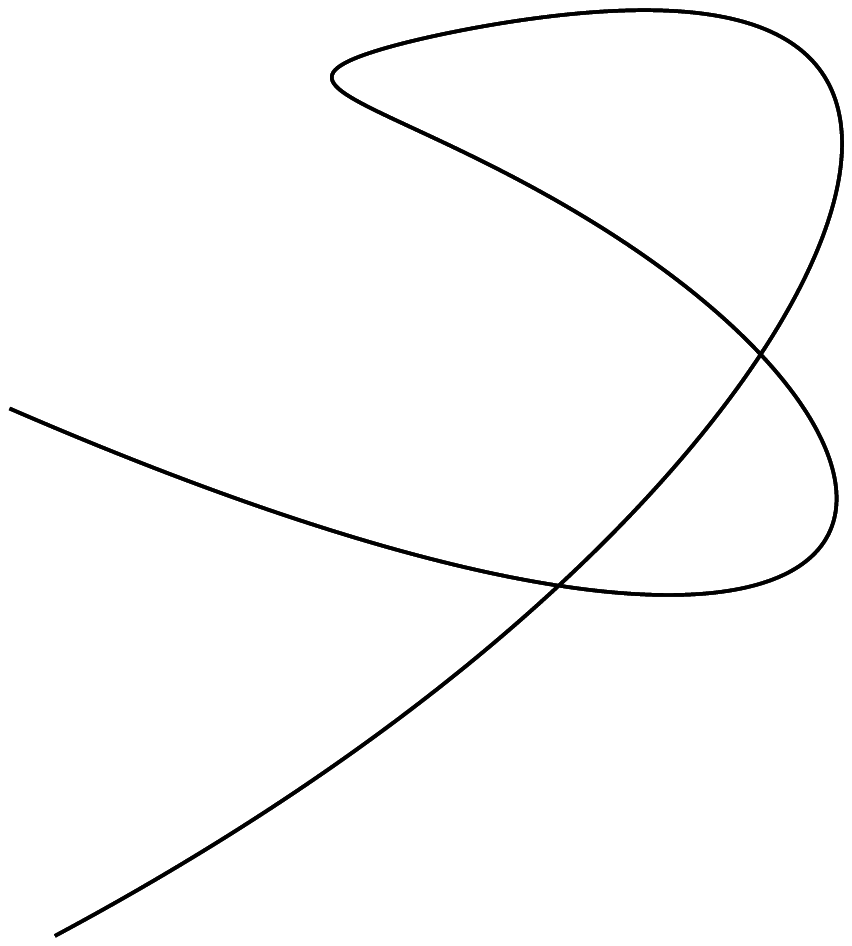} \\ \hline
\end{tabular}\caption{Implicit plots of the approximations of the degree seven B\'ezier curve pictured in Figure \ref{fig:param}, for implicit degrees $m=4,5,6$ and $7,$ in the monomial, Bernstein, Lagrange and Chebyshev bases.}\label{fig:implicitizations}
\end{center}
\end{figure}

As discussed previously, minimizing the algebraic error does not necessarily minimize the geometric distance between the implicit and parametric curves. In order to visually compare how the methods perform in terms of geometric approximations, Figure \ref{fig:implicitizations}  plots the implicit approximations of the parametric curve pictured in Figure \ref{fig:param}. The curve was chosen as an example that represents the general approximation properties of each of the bases. However, it should be noted that different examples can give different results.

We see that for the quartic approximation, the Lagrange and Chebyshev methods are already performing fairly well with only some detail lost close to the double point singularities. Despite exhibiting several intersections with the parametric curve, the Bernstein method gives little reproduction of the shape. The monomial approximation bears almost no resemblance to the original curve. For the quintic approximation, the Chebyshev and Lagrange bases again perform very similarly, giving excellent approximations that replicate the singularities well. These approximations would be sufficient for many applications. The Bernstein method performs similarly to the Chebyshev and Lagrange approximations of degree four, with only some loss of detail at each of the double points. Again the monomial basis gives almost no replication of the curve. It is also interesting to note the presence of extraneous branches visible in the Bernstein, Lagrange and Chebyshev approximations at degree five. This is a feature which may occur with any of the methods. At degree six the Bernstein, Lagrange and Chebyshev methods all give excellent results over the entire interval. The monomial method is beginning to show good approximation at the centre of the interval, however, this deteriorates towards the ends. At degree seven we expect exact results, up to numerical stability, for all of the algorithms. Visually, the implicitizations in all of the bases agree very closely. 

For degree seven, we can also perform the Lagrange method in exact precision as described in Section \ref{ssec:exact}. Using this method we obtain a null space of dimension one, which confirms that the correct algebraic degree is in fact seven. We may also use this exact implicitization to compare the relative errors $e_\alpha$ (using the infinity norm) for the implicitizations given by the different bases, as follows:
\begin{eqnarray*}
e_\text{mono} & = & 1.17\times10^{-8}, \\
e_\text{bern} & = & 7.46\times10^{-11}, \\
e_\text{lag} & = & 2.01\times10^{-4}, \\
e_\text{cheb} & = & 1.11\times10^{-5}.
\end{eqnarray*}
The numerical stability properties of the Bernstein basis are well documented in mathematical literature (see for example \cite{farouki_1996}). It appears that these properties are also preserved by the implicitization algorithm presented here, with the Bernstein basis outperforming the other bases to some orders of magnitude. In relation to the numerical stability of the methods, it is interesting to note the distribution of singular values given by the singular value decomposition of the $\mathbf{D}_{\alpha}$ matrices \cite{schicho_2005}. For this comparison we normalize the singular values to all lie in $[0,1].$ The Bernstein basis has one singular value close to zero in double precision; the second smallest being approximately $7.74\times10^{-9}.$ This shows that the solution is quite unique. The case is similar in the monomial basis, however, the second smallest singular value is approximately $6.24\times10^{-10}.$ For the Lagrange and Chebyshev bases the second smallest singular values are $2.54\times10^{-15}$ and $1.17\times10^{-14}$ respectively. Such close proximity between the smallest singular values leads to issues with numerical stability. However, it can also be useful since, as discussed previously, the singular vectors corresponding to the larger singular values are also candidates for approximation. Thus we would expect the second best approximation in the Chebyshev and Lagrange bases to be much better than the second best approximation in the Bernstein basis.

It is also interesting to note that when attempting to use the weak method for approximate implicitization as an exact method here, we obtain a completely different solution, with relative error given approximately by $e_\text{weak} = 0.607.$ This seems to be due to the fact that the nine smallest eigenvalues (which are equal to the singular values, since $\mathbf{M}$ is symmetric positive semi-definite), are all below machine precision. Thus, the solution given by the weak method could be a combination of any of the nine corresponding eigenvectors. However, despite the large relative error, the weak method still returns a solution which gives a reasonable geometric approximation.

Typical algebraic error distributions obtained from the methods in this section are displayed in Figure \ref{fig:err}. A more accurate approximation of the geometric error can be given by dividing the algebraic error by the norm of the gradient of the given implicit representation. However, since the methods we present do not control the gradient, we have not included such plots here. 

\begin{figure}
\center
\subfloat[Monomial]{\label{fig:monoerr}\includegraphics[scale=0.45]{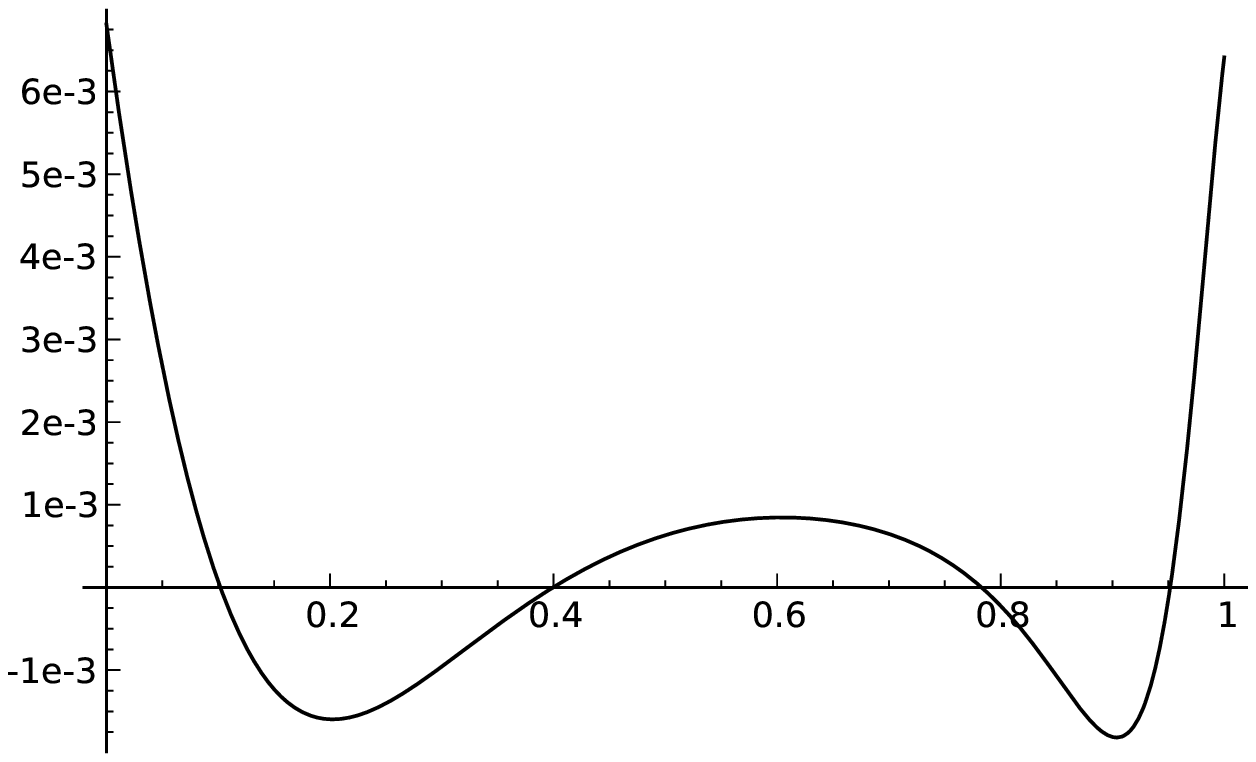}}
\subfloat[Bernstein]{\label{fig:bernerr}\includegraphics[scale=0.45]{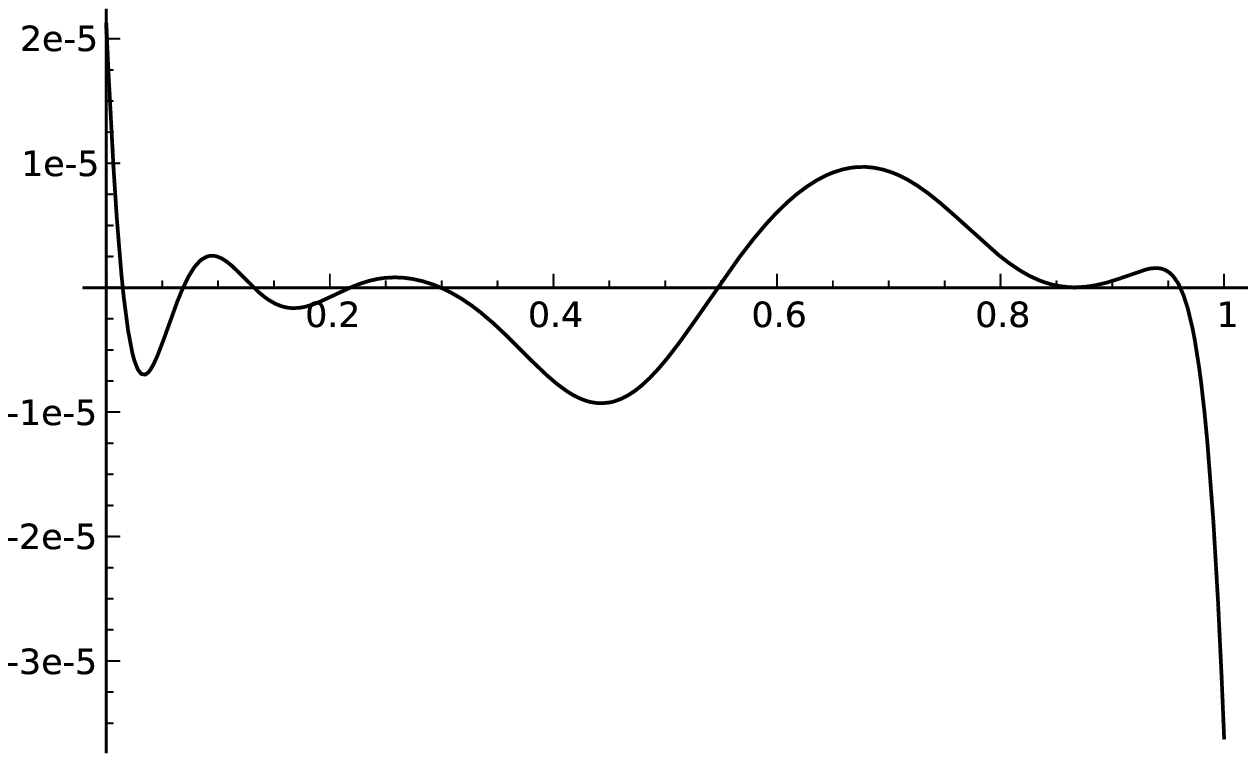}} \\
\subfloat[Lagrange]{\label{fig:lagerr}\includegraphics[scale=0.45]{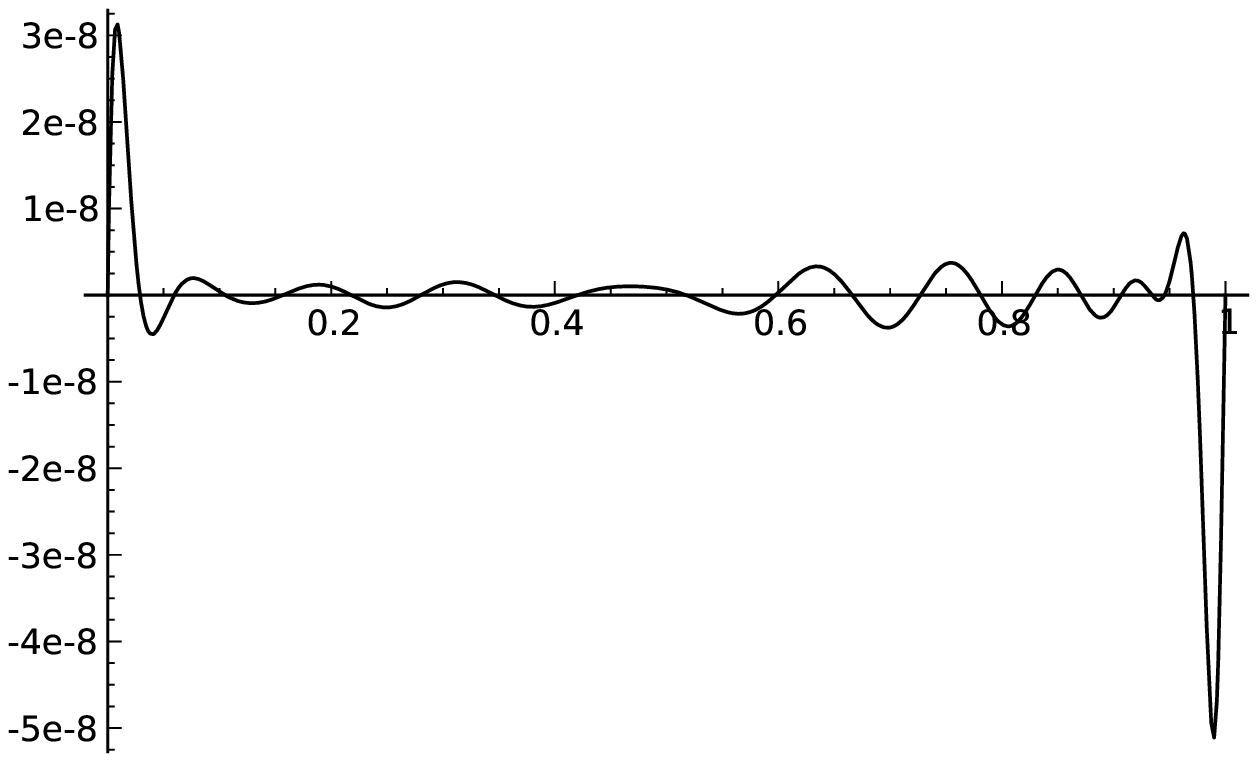}}
\subfloat[Chebyshev]{\label{fig:cheberr}\includegraphics[scale=0.45]{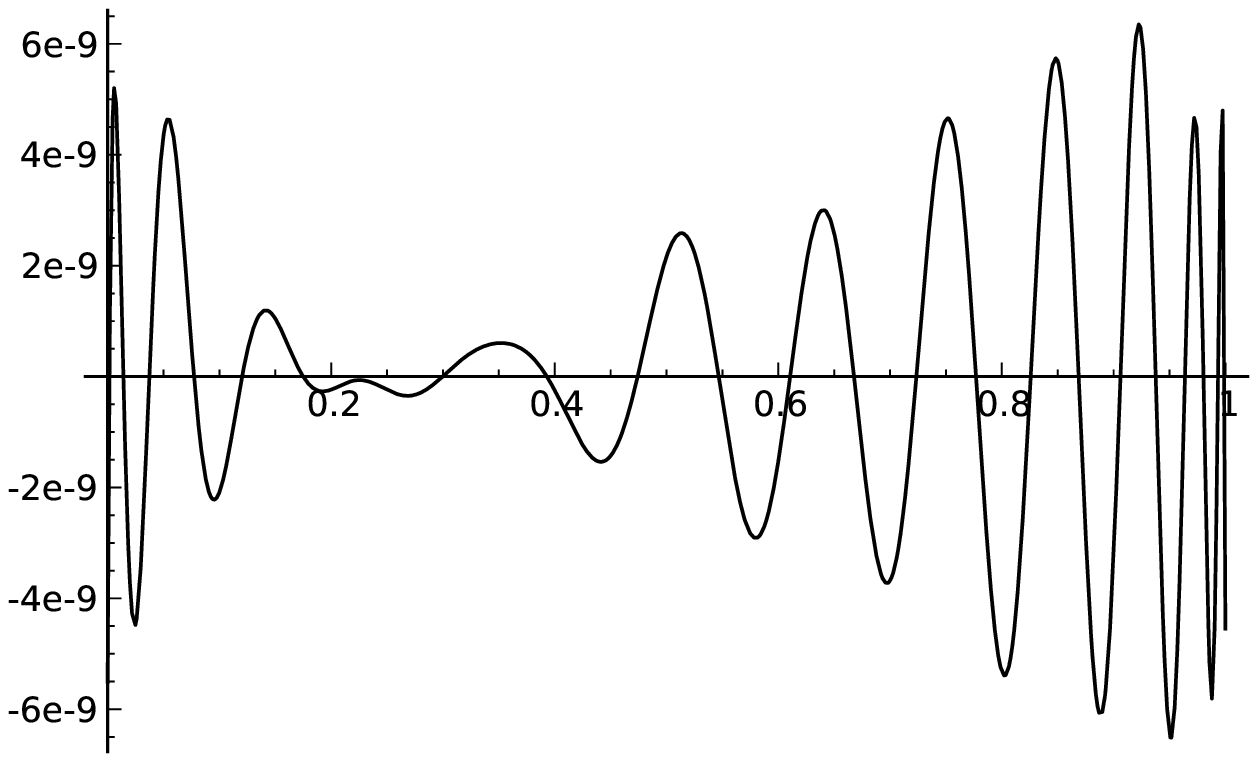}}
\caption{Typical algebraic error distributions $q(\mathbf{p}(t)),$ for the different bases. These are taken from the quintic implicit approximations of the parametric curve in Figure \ref{fig:param}.}\label{fig:err}
\end{figure}

\subsection{Discussion}

In the example of section \ref{ssec:vis}, it is striking how well the Lagrange basis performs in relation to the Chebyshev basis. Despite the spike in error near the ends of the interval, the geometric approximation appears to remain relatively good; and in some cases, even better than the Chebyshev. Thus, the comparisons in this section may lead to different conclusions as to which is the best algorithm to choose in general. It is clear that the monomial basis performs relatively badly, but the other bases all tend to perform well. The speed and simplicity of the Lagrange method is appealing, whereas the stability of exact implicitization in the Bernstein basis is desirable for many applications. The fact that the Legendre and Chebyshev methods have the guarantee of solving a least squares problem (see Theorem \ref{thm:DTD}), and that there exist efficient algorithms for computing them, also gives justification for choosing them as general methods. As an overview of this qualitative comparison, we display various aspects of the implementations in Table \ref{tab:qual}. The table summarizes how the algorithms perform in terms of the  stability, generality and what sort of problems they solve (i.e., least squares or uniform approximations).

\begin{table}
\begin{center}
  \begin{tabular}{ | l || c | c | c | c | c | c | c | }
    \hline
               & Least squares   & Uniform    & Stability & Generality     \\ \hline
    Lagrange   & Good            & OK         & OK        & Any            \\ \hline
    Legendre   & Very good       & Good       & OK        & Rational       \\ \hline
    Chebyshev  & Very good       & Very Good  & OK        & Rational       \\ \hline
    Bernstein  & OK              & OK         & Very Good & Rational       \\ \hline
    Weak       & Very good       & OK         & Very Bad  & Integrable     \\ \hline
  \end{tabular}\caption{A qualitative comparison of the algorithms. The least squares, and uniform columns refer to how well the algorithms perform in terms of producing such approximations in the algebraic error function.}\label{tab:qual}
\end{center}
\end{table}

One undesirable property of approximate implicitization is the possibility of introducing new singularities that are not present in the parametric curve. As the implicit polynomial representation is global, we cannot control what happens outside the interval of approximation. In particular, there could appear self-intersections of the curve within the interval of interest. This is an artifact that can appear using any of the methods described in this paper. However, such problems can be avoided by adding constraints to the algebraic approximation \cite{sederberg_1999}, or by using information about the gradient of the implicit curve in the approximation \cite{juettler_2002}. In general, adding such constraints will reduce the convergence rates of these methods \cite{dokken_2001}.

The computation times for each of the methods varies. In all the current implementations of the methods, the matrix generation is the dominant part of the algorithm, and the SVD is generally fast. When constructing the matrices, the monomial and Bernstein methods suffer from computationally expensive expansions for high degrees, whereas the Chebyshev and Lagrange methods are based on point sampling and FFT, which can be implemented in parallel. Computational features of the methods will be the subject of further research, including exploiting the parallelism of GPUs to enhance the algorithms.

\section{Approximate implicitization of surfaces}\label{sec:surfaces}

In this section we will discuss how the methods presented for curves in the preceding sections generalize to surface implicitization. We will also provide a visual example of approximate implicitization of surfaces.

A parametric surface in $\mathbb{R}^3$ is given by $\mathbf{p}(s,t) = (p_1(s,t),p_2(s,t),p_3(s,t))$ where $p_1, p_2$ and $p_3$ are functions in parameters $(s,t)\in\Omega\subset\mathbb{R}^2.$  Similarly to curves, in $\mathbb{P}^3,$ we use the homogeneous form of such a surface to write 
\[
\mathbf{p}(s,t) = (g_1(s,t),g_2(s,t),g_3(s,t),h(s,t)),
\]
for bivariate polynomials $g_1, g_2, g_3$ and $h.$ 

Although we have the option of using tensor-product polynomials for the implicit representation, here we choose polynomials of total degree $m.$ An implicit polynomial $q$ of total degree $m$ can be described in a basis $(q_k(\mathbf{u}))_{k=1}^M,$ where $M=\binom{m+3}{3},$ with coefficients $\mathbf{b} = (b_k)_{k=1}^M.$ The choice of using the Bernstein basis for the implicit polynomial can be extended by considering a barycentric coordinate system defined over a tetrahedron containing the parametric surface. For surfaces in $\mathbb{P}^3,$ the homogeneous Bernstein basis is given by
\[
q_{\mathbf{k}}(u,v,w,z) = \binom{m}{\mathbf{k}} u^{k_1} v^{k_2} w^{k_3} z^{k_4}, \quad \text{ for } |\mathbf{k}| = k_1+k_2+k_3+k_4 = m,
\]
where $u,v,w$ and $z$ denote the homogeneous coordinates, and $\mathbf{k}=(k_1,k_2,k_3,k_4)$ denotes a multi-index. Again, this basis forms a partition of unity and we order it by letting $q_k$ correspond to $q_{\mathbf{k}},$ for $k=1,\ldots,M,$ where $k = k(\mathbf{k})$ denotes lexicographical ordering.

When applying the original algorithm for approximate implicitization, we observe that the expression $q\circ\mathbf{p}$ is a bivariate polynomial in $s$ and $t.$ As such, we can write the function in a basis $\bm\alpha(s,t)=(\alpha_j(s,t))_{j=1}^L$ as
\begin{equation}\label{eq:biimp}
q(\mathbf{p}(s,t)) = \bm\alpha(s,t)^T\mathbf{D}_\alpha\mathbf{b}.
\end{equation}
The description of $\bm\alpha(s,t)$ and the number of basis functions $L$ is dependent on the type of parametric surface. We thus make distinctions for the two most interesting cases - tensor-product surfaces, and surfaces on triangular domains - in the following subsections.

The weak method presented in Section \ref{sec:weak} can also be generalized to surfaces. The problem can be stated as 
\begin{equation}\label{eq:lsprob}
\min_{\Vert\mathbf{b}\Vert_2 = 1} \int_{\Omega} w(s,t) q(\mathbf{p}(s,t))^2 \text{ d}s\text{ d}t,
\end{equation}
where $w(s,t)$ is some weight function defined on the domain of approximation $\Omega.$ In this way, we obtain a linear algebra problem as before, where the solution is given by the eigenvector corresponding to the smallest eigenvalue of a matrix $\mathbf{M}_w=(m_{k,l})_{k=1,l=1}^{M,M},$ defined by 
\begin{equation}\label{eq:Msurf}
m_{k,l} =  \int_\Omega w(s,t) q_k(\mathbf{p}(s,t)) q_l(\mathbf{p}(s,t)) \text{ d}s \text{ d}t.
\end{equation}

\subsection{Tensor-product parametric surfaces}

For rational tensor-product surfaces of bidegree $\mathbf{n}=(n_1,n_2),$ we can write, 
\[
\mathbf{p}(s,t) = \sum_{i_1=0}^{n_1}\sum_{i_2=0}^{n_2} \mathbf{c}_{i_1,i_2} \phi_{i_1}(s) \psi_{i_2}(t), 
\]
where $(\phi_{i_1})_{i_1=0}^{n_1}$ and $(\psi_{i_2})_{i_2=0}^{n_2}$ are bases for univariate polynomials of respective degree $n_1$ and $n_2,$ the domain $\Omega=[a,b]\times[c,d]$ is the Cartesian product of two univariate intervals, and $\mathbf{c}_{i_1,i_2}\in\mathbb{P}^3$ for $i_1=0,\ldots,n_1$ and $i_2=0,\ldots,n_2.$ In this case, equation (\ref{eq:biimp}) can be written in a tensor-product basis $\bm\alpha(s,t) = (\alpha_{j_1}(s)\beta_{j_2}(t))_{j_1=1,j_2=1}^{L_1,L_2}$ of bidegree $m\mathbf{n} = (mn_1,mn_2)$ as follows:
\[
q(\mathbf{p}(s,t)) = \sum_{k=1}^M b_k \sum_{j_1=1}^{L_1}\sum_{j_2=1}^{L_2} \hat{d}_{(j_1,j_2),k} \alpha_{j_1}(s)\beta_{j_2}(t),
\]
where $L_1=mn_1+1$ and $L_2=mn_2+1.$ We must use an ordering for the indices $(j_1,j_2)$ in order to enter the coefficients $\hat{d}_{(j_1,j_2),k}$ in matrix form. Again we choose the lexicographical ordering $j=j(j_1,j_2) = (j_1-1)L_2+j_2,$ so that
\begin{equation}\label{eq:factsurf}
\mathbf{D}_{\alpha,\beta} = (d_{j,k})_{j=1,k=1}^{L,M},
\end{equation}
where $L = L_1L_2 = m^2n_1n_2+mn_1+mn_2+1$ and $d_{j,k} = \hat{d}_{(j_1,j_2),k}.$ The algorithm then proceeds as for curves by using the singular value decomposition and selecting the vector corresponding to the smallest singular value. 

The univariate bases $\bm\alpha=(\alpha_{j_1})_{{j_1}=1}^{L_1}$ and $\bm\beta=(\beta_{j_2})_{j_2=1}^{L_2}$ can be be chosen arbitrarily, as in the case for curves. In this way, Theorem \ref{thm:DTD}, Propositions \ref{prop:lagrange_convergence}, \ref{prop:cheb_lagrange_convergence}, \ref{prop:bernstein_convergence}, \ref{prop:samps} and Corollary \ref{cor:bern_con} from the univariate case, all carry over to the tensor-product case with minimal effort\footnote{For a tensor-product extension of Proposition \ref{prop:samps}, the samples should be taken on a tensor-product grid and the number of samples is given by $K=(2n_1n_2^2+1)(2n_1^2n_2+1)$.}. This applies also to higher dimensional tensor-product hypersurfaces. As an example, consider Theorem \ref{thm:DTD}. In the tensor-product case, we have the following, almost verbatim to the univariate case:
\begin{thm}
Let $\bm\alpha$ be a tensor-product polynomial basis, orthonormal with respect to the given inner product $\langle\cdot,\cdot\rangle_w,$ and $\mathbf{D_{\alpha,\beta}}$ and $\mathbf{M}_w$ be defined by (\ref{eq:factsurf}) and (\ref{eq:Msurf}) respectively. Then $\mathbf{M}_w = \mathbf{D}^T_{\alpha,\beta}\mathbf{D_{\alpha,\beta}}.$
\end{thm}

\begin{proof}
Similar to Theorem \ref{thm:DTD}, with the relevant inner product given by 
\[
\langle f,g \rangle_w = \int_\Omega w(s,t) f(s,t) g(s,t) \text{ d}s \text{ d}t,
\]
for real bivariate polynomials $f,g.$ 
\end{proof}

\subsection{Triangular surfaces}

For rational surfaces of total degree $n$ on a triangular domain $\Omega,$ we can write, 
\[
\mathbf{p}(s,t) = \sum_{i=1}^{N} \mathbf{c}_{i} \phi_i(s,t), 
\]
where $(\phi_i)_{i=1}^{N}$ is a basis for bivariate polynomials of total degree $n$ with $N=\binom{n+2}{2},$ and $\mathbf{c}_i\in\mathbb{P}^3$ for $i=1,\ldots,N.$ In this case, equation (\ref{eq:biimp}) can be written in a bivariate basis $\bm\alpha=(\alpha_j)_{j=1}^L$ of total degree $mn$ as follows:
\[
q(\mathbf{p}(s,t)) = \sum_{k=1}^M b_k \sum_{j=1}^{L} d_{j,k} \alpha_j(s,t),
\]
where, $L=\binom{mn+2}{2}.$ Lexicographical ordering on the respective degrees of $s$ and $t$ in the basis $\bm\alpha$ can be used to enter the coefficients in matrix form, $\mathbf{D}_{\alpha} = (d_{j,k})_{j,k=1,1}^{L,M}.$ 

Surfaces on triangular domains may be considered a more fundamental generalization than tensor-product surfaces, however, they often exhibit several difficulties not present in the tensor-product case. For example, most practical applications of the weak method of Section \ref{sec:weak} use numerical quadrature, a method which is difficult to implement on triangular domains for high degrees. Since the degree of the integrand in the weak implicitization of a surface of degree $n$ has degree $2mn,$ high degree quadrature rules are required. For example, exact implicitization of a general quintic parametric surface, with implicit degree $m=n^2=25$ would need a quadrature rule accurate to order 250. Although it would be wise to use a lower degree approximation in this case (e.g., $m=5$), the degree of the integrand would still be too high for many quadrature rules on triangular domains. High degree quadrature rules can be constructed by first transforming the domain into a square, and using two univariate quadrature rules. However, these rules are not symmetric in the triangle and require more samples than is mathematically necessary \cite{lyness_1994}.

Certain methods for approximate implicitization are, however, easy to generalize. For example, the Bernstein basis has a natural representation on simplex domains using barycentric coordinates, and thus the use of approximate implicitization on triangular surfaces in this basis is simple \cite{barrowclough_2010}. The convergence result of Corollary \ref{cor:bern_con} can be extended to this case, together with well established degree elevation algorithms, in order to obtain better approximation results. 

The Lagrange basis method from Section \ref{ssec:lag}, which was based on sampling, can also be extended to triangular surfaces. However, when choosing samples, it is essential that the Lagrange polynomials defining the $\bm\alpha$-basis are linearly independent; that is, that they do in fact form a basis. For curves, choosing unique parameters for the sample points was sufficient (see Proposition \ref{prop:samps}). However, for surfaces we must add the requirement that the samples must not all lie on a curve of degree $mn$ in the parameter domain, where the number of samples is given by $L=\binom{mn+2}{2}.$

Orthogonal polynomials on triangular domains also exist and an extension of Theorem \ref{thm:DTD} holds in this case. However, fast algorithms for generating the coefficeints appear to be difficult to replicate, thus limiting the applicability of this method in practice. In particular, there appears to be no direct analogue of the FFT method for generating Chebyshev coefficients. We refer the reader to \cite{farouki_2003} for a discussion of orthogonal polynomials on triangular domains.

\subsection{An example of surface implicitization}\label{ssec:teapot}

As an example of approximate implicitization of tensor-product surfaces we will consider the problem of approximating the well known Newells' teapot model. It is stated in \cite{sederberg_1995} that ``the 32 bicucic patches defining Newells' teapot provide a surprisingly diverse set of tests for moving surface implicitization''. In that paper, properties of the moving surface implicitization algorithm for the different patches are discussed and exact implicit degrees for each patch are stated. We will consider the same 32 patches here, but instead use approximate methods, where the degree of approximation has been chosen manually to give good visual results.

Each of the 32 bicubic parametric surfaces has been approximated using the tensor-product Chebyshev method and the degrees stated in Table \ref{tab:teapotdeg}. The resulting surface has been ray traced in Figure \ref{fig:teapot} using POV-Ray \cite{povray}, and clearly gives an excellent implicit approximation to the parametric teapot model. Moreover, the degrees of the surface patches are greatly reduced from the exact degrees, which are also stated in Table \ref{tab:teapotdeg}. The highest degree patch in the approximated model is six, whereas if exact methods are used patches of degree up to 18 are required. Generally, it appears that for visual purposes, the degree can be reduced to roughly one third of the exact degree. 

This example shows one potential application of approximate implicitization, however, there are several factors that should be noted. Firstly, a significant amount of user input was required to generate the approximations of the teapot patches. This involved choosing degrees that were suitable for each patch, and also choosing approximations without extra branches in the region of interest. This was done by considering approximations corresponding to other singular values than the smallest. For example, for the upper handle patches we chose the approximation corresponding to the fourth smallest singular value. For each increased singular value, the convergence rate of the method is reduced by one \cite{dokken_2006}. However, even with the reduced convergence rates, the Chebyshev method continues to give excellent approximations. User input was also given to define the clipping region for the ray tracing of each patch. In this example, the clipping regions were boxes parallel to the $xy,xz$ and $yz$-planes. however, in more complicated examples it is not always suitable to use box regions for this purpose. 

Another feature of this example is that the continuity between the parametric patches has been approximated very well in the implicit model. This is mainly due to the high convergence rates, which give good approximations over the entire surface region. However, in this case, there is also symmetry in the model meaning the edge curves where the patches meet can be approximated in a symmetric way. To achieve this we have used the monomial basis for the implicit representation since it is symmetric around the $z$-axis. For more general models, such symmetry would not be possible.

\begin{figure}
\begin{center}
\includegraphics[scale=0.1]{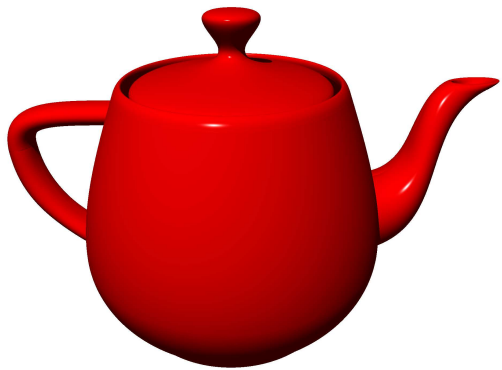}\caption{Teapot defined by 32 approximately implicitized patches from Section \ref{ssec:teapot}, with degrees given in Table \ref{tab:teapotdeg} and ray traced using POV-Ray \cite{povray}.}\label{fig:teapot}
\end{center}
\end{figure}

\begin{table}
\begin{tabular}{ |c|c|c| }
\hline
& Exact degree & Approximate degree \\ \hline
4 $\times$ rim          & 9 & 4 \\ \hline
4 $\times$ upper body   & 9 & 3 \\ \hline
4 $\times$ lower body   & 9 & 3 \\ \hline
2 $\times$ upper handle & 18 & 4 \\ \hline
2 $\times$ lower handle & 18 & 4 \\ \hline
2 $\times$ upper spout  & 18 & 5 \\ \hline
2 $\times$ lower spout  & 18 & 6 \\ \hline
4 $\times$ upper lid    & 13 & 3 \\ \hline
4 $\times$ lower lid    & 9 & 4 \\ \hline
4 $\times$ bottom       & 15 & 3 \\ \hline
\end{tabular}\caption{Exact implicit degrees of the 32 Newells' teapot patches and the degrees used for approximate implicitization in Section \ref{ssec:teapot}.}\label{tab:teapotdeg}
\end{table}

\section{Conclusions}

We have presented and unified several new and existing methods for approximate implicitization of rational curves using linear algebra. Theoretical connections between the different methods have been made together with qualitative comparisons. Extensions of the methods to both tensor-product and triangular surfaces have been discussed. By considering various issues such as approximation quality and computational complexity, we regard the Chebyshev and Legendre methods as the algorithms of choice for approximation of most rational parametric curves. However, to obtain good numerical stability when using floating point arithmetic for exact implicitization, the Bernstein basis is a more favourable choice. Future research could include how the methods can be improved, for example, by exploiting sparsity as in \cite{emiris_2005}, or by adding continuity constraints to the approximations \cite{bajaj_1992}.

\section{Acknowledgements}

The research leading to these results has received funding from the [European Community's] Seventh Framework Programme [FP7/2007-2013] under grant agreement n$^\circ$ [PITN-GA-2008-214584], and from the Research Council of Norway through the IS-TOPP program.

\bibliographystyle{plain}
\nocite{*}
\bibliography{paper}

\end{document}